\documentclass[11pt,letterpaper]{amsart}
\usepackage[utf8]{inputenc}
\usepackage{amsmath}
\usepackage{amssymb}
\usepackage{amsthm}
\usepackage{mathrsfs}
\usepackage{mathtools}
\usepackage[dvipsnames]{xcolor}
\usepackage{tikz}
\usepackage{hyperref}
\usepackage[nameinlink,capitalize]{cleveref}
\newcommand\myshade{85}
\colorlet{mylinkcolor}{blue}
\colorlet{mycitecolor}{red}
\colorlet{myurlcolor}{Aquamarine}

\hypersetup{
  linkcolor  = mylinkcolor!\myshade!black,
  citecolor  = mycitecolor!\myshade!black,
  urlcolor   = myurlcolor!\myshade!black,
  colorlinks = true,
}
\usepackage{tikz-cd}
\usepackage{mathtools}
\usepackage{verbatim}
\usepackage{appendix}

\usepackage{longtable}
\usepackage{url}
\usepackage{wrapfig}

\allowdisplaybreaks

\newtheorem{thm}{Theorem}[section]
\crefname{thm}{Theorem}{Theorems}
\newtheorem{cor}[thm]{Corollary}
\newtheorem{prop}[thm]{Proposition}
\crefname{prop}{Proposition}{Propositions}
\newtheorem{lem}[thm]{Lemma}
\crefname{lem}{Lemma}{Lemmas}
\newtheorem{clm}[thm]{Claim}
\newtheorem{conj}[thm]{Conjecture}
\newtheorem{quest}[thm]{Question}
\newtheorem{defn}[thm]{Definition}

\theoremstyle{definition}
\crefname{defn}{Definition}{Definitions}

\crefname{prin}{Principle}{Principles}
\newtheorem{exmp}[thm]{Example}

\theoremstyle{definition}
\newtheorem{rmk}[thm]{Remark}

\newtheorem*{ack*}{Acknowledgements}

\newcommand{\Gr}{\operatorname{Gr}}

\newcommand{\rk}{{\operatorname{rk}}}
\newcommand{\crk}{{\operatorname{crk}}}

\title{Log-concavity of matroid h-vectors and mixed Eulerian numbers}
\author{Andrew Berget,  Hunter Spink, Dennis Tseng}
\date{}

\begin{document}
\begin{abstract}For any matroid $M$, we compute the Tutte polynomial $T_M(x,y)$ using the mixed intersection numbers of certain classes in the combinatorial Chow ring $A^\bullet(M)$ arising from hypersimplices. Using the mixed Hodge-Riemann relations, we deduce a strengthening of the log-concavity of the $h$-vector of a matroid complex, improving on an old conjecture of Dawson. \end{abstract}

\maketitle

\section{Introduction}
Matroids are a common generalization of graphs and vector configurations, encoding a related notion of ``independence''.
In their celebrated paper, Adiprasito, Huh, and Katz \cite{AHK18}
showed that the coefficients of the characteristic polynomial $\chi_M(q)$  form a log-concave
sequence for any matroid $M$, proving a conjecture of Rota-Heron-Welsh
\cite{R71}. Building on previous work \cite{M93,FY04,HK12},
log-concavity follows from two ingredients:
\begin{enumerate}
\item A ``combinatorial Chow ring'' $A^{\bullet}(M)$ associated to $M$ satisfying
  Hodge-Riemann type relations analogous to the Hodge index theorem
  for surfaces.
\item A description of $\chi_M(q)$ in terms of mixed intersections of
  ``nef’’ divisors $\alpha_M$ and $\beta_M$.
\end{enumerate}
There has been progress in simplifying and extending the first
ingredient
\cite{BES19,BHMPW20,ADH20}. However, there has been comparatively little
research trying to compute new log-concavity statements from the
established Hodge-Riemann relations in $A^\bullet(M)$. We will focus on expanding on the second ingredient, and the resulting
log-concavity statements.


 Our main result concerns a conjecture of Dawson {\cite{Dawson}} on the independence complex $\operatorname{IN}(M)$, which has as its simplices the independent sets of $M$. A classical description of the $h$-vector $h_0,\ldots,h_{\rk(M)}$ of $\operatorname{IN}(M)$ (see \Cref{cor:nointernal}) establishes that this sequence is nonnegative, and has no internal zeros (meaning if $i<j<k$ and $h_i,h_k\ne 0$, then $h_j\ne 0$).
\begin{conj}[{\cite{Dawson}}]\label{conj:Dawson}
The $h$-vector $h_0,\ldots,h_{\rk(M)}$ of $\operatorname{IN}(M)$ is a log-concave sequence, meaning $h_i^2\ge h_{i-1}h_{i+1}$ for $0<i<\rk(M)$. In particular, the sequence is unimodal, meaning there exists $k$ such that $h_0\le \cdots \le h_k$ and $h_k\ge h_{k+1}\ge \cdots \ge h_{\rk(M)}$.
\end{conj}
Huh \cite{HuhAdvances} established this conjecture for matroids realizable over characteristic $0$, which includes graphic matroids. Using the mixed Hodge-Riemann relations in $A^\bullet(M)$ \cite[Theorem 8.9]{AHK18}, we strengthen Dawson's conjecture.
\begin{thm}\label{preintrothm}
Let $M$ be an arbitrary matroid with rank $\rk(M)$ and corank $\crk(M)$. Let $h_0,\ldots,h_{\rk(M)}$ be the $h$-vector of the matroid complex of $M$. Then when $\crk(M)\ge 2$ we have
$$h_i^2-h_{i-1}h_{i+1}\ge \frac{1}{\crk(M)-1}(h_i-h_{i-1})(h_{i+1}-h_i).$$
In particular, for any matroid $M$ we have $h_i^2\ge h_{i-1}h_{i+1}$ for all $i$, resolving \Cref{conj:Dawson}.
\end{thm}

We make a few remarks about this theorem.
\begin{enumerate}
\item In the trivial cases where $\crk(M)<2$, we note that $\crk(M)=0$ implies the $h$-vector is $(0,\ldots,0,1)$ and $\crk(M)=1$ implies it is equal to a sequence of the form $(0,\ldots,0,1,\ldots,1,0,\ldots,0)$.
    \item One could deduce log-concavity of the $h$-vector sequence from the ``strengthened log-concavity'' by noting that adding loops to $M$ increases $\crk(M)$ but preserves the $h$-vector. However it turns out that for $a,b,c$ a sequence of nonnegative numbers with no internals zeros, the inequality $b^2-ac \ge \lambda (b-a)(c-b)$ for any $\lambda \le 1$ implies $b^2-ac \ge 0$ directly (see \Cref{cor:lc}).
    \item As this result shows that $h_i$ are unimodal, the ``strengthened log-concavity'' inequality implies that for nonzero $h_{i-1},h_i,h_{i+1}$ we have strict log-concavity $h_i^2>h_{i-1}h_{i+1}$ except when $h_{i-1}=h_i=h_{i+1}=\max\{h_j\}$, which happens for infinitely many rank $2$ matroids and many higher rank matroids \cite{DLKK12}. However, the inequality in \Cref{preintrothm} can be an equality even when $(h_i-h_{i-1})(h_{i+1}-h_i)>0$, showing that the result is in some sense sharp.
    \item The resolution of Dawson's conjecture was announced by Ardila, Denham and Huh in 2017 (see \cite{ardilaGeometry}). The present paper appeared concurrently to \cite{ADH20}, which develops a new ``conormal Chow ring'' associated to  what they call a ``bipermutohedral fan''. Their paper also establishes a conjecture of Brylawski on the log-concavity of the $h$-vector of the no broken circuit complex of $M$ through their conormal Chow ring. The combinatorics of the bipermutohedral fan were later simpified and extended in \cite{BEST} through the framework of ``tautological classes of matroids''.
\end{enumerate}

For $G=(E,V)$ a graph and $M=M(G)^*$ the cocycle matroid of $G$ (the dual to the associated graphic matroid), the result can be phrased in terms of the reliability polynomial $R_G(p)$. Recall that $R_G(p)$ is the probability that if each edge of $G$ is discarded with probability $1-p$ that each of the $\kappa(G)$ connected components of $G$ remain connected. The $h$-sequence of the reliability polynomial is defined by
$$R_G(p)=p^{\crk(M)}\sum_{i=0}^{\rk(M)} h_i(1-p)^i$$
where $\crk(M)=|V|-\kappa(G)$ and  $\rk(M)=|E|-|V|+\kappa(G)$. \Cref{preintrothm} appears to be new even in this case, strengthening the result in \cite{HuhAdvances}.
\subsection{Matroid invariants and hypersimplices}
Let $M$ be a rank $r+1$ loopless matroid on ground set $\{0,\ldots,n\}$.
The ring $A^\bullet(M)$ is equipped with a degree map $\deg:A^r(M)\to \mathbb{Z}$, such that for each generalized permutahedron $P\subset \mathbb{R}^{n+1}$ (an integral polytope all of whose edges are parallel to $e_i-e_j$ for various $i,j$), there is an associated element $L\in A^1(M)_{\mathbb{R}}$, and if $L_1,\ldots,L_r$ are associated to generalized permutohedra $P_1,\ldots,P_n$, then $\deg_{A^\bullet(M)}(L_1\cdots L_r)$ is a nonnegative integer associated to the matroid $M$.

The mixed Hodge-Riemann relations in $A^\bullet(M)$ then give us log-concavity statements.
\begin{thm}[{\cite[Theorem 8.9]{AHK18}}]
Let $L_1,\ldots,L_{r-2},H,K\in A^1(M)_{\mathbb{R}}$ be elements associated to generalized permutahedra. Let $\Omega=L_1\cdots L_{r-2}$. Then $$\deg_{A^\bullet(M)}(\Omega\cdot H^2),\deg_{A^\bullet(M)}(\Omega \cdot HK),\deg_{A^\bullet(M)}(\Omega \cdot K^2)$$
is a log-concave sequence.
\end{thm}
A natural way to ensure these degrees are matroid invariants (only depending on the isomorphism type of $M$) is to only consider generalized permutahedra invariant under the permutation action on $\mathbb{R}^{n+1}$.

For example, the classes $\alpha,\beta$ used in \cite{HK12,H14,AHK18} to establish log-concavity for the coefficients of the characteristic polynomial $\chi_M(q)$ correspond to the simplices
$P_{n,n+1}=[0,1]^{n+1}\cap \{\sum x_i=n\}$ and $P_{1,n+1}=[0,1]^{n+1}\cap \{\sum x_i=1\}$ respectively. It turns out that every $S_{n+1}$-invariant generalized permutahedron is a Minkowski sum of nonnegative scalings of the $n$-dimensional hypersimplices $$P_{k,n+1}=[0,1]^{n+1}\cap \{\sum x_i=k\},$$
and writing $\gamma_i\in A^1(M)$ for the element associated to $P_{i,n+1}$, the class associated to any such polytope is therefore a non-negative linear combination of $\beta=\gamma_1,\gamma_2,\ldots,\gamma_{n-1},\gamma_n=\alpha$. By using linearity of the degree, the search for invariants of matroids and associated log-concavity properties arising in this way is equivalent to understanding the following question.
\begin{quest}
\label{quest:whatinvariants}
What matroid invariants arise as the numbers
$$\deg_{A^\bullet(M)}(\gamma_1^{a_1}\cdots \gamma_n^{a_n})$$
for $a_1,\ldots,a_n\ge 0$ with $\sum a_i=r$?
\end{quest}
Recall that the Tutte polynomial $T_M(x,y)$ is the universal deletion contraction invariant for matroids \cite{Ardila}, which has as one of its many interesting specializations $\chi_M(q)$.
When the multiplicity support in $\gamma_1^{a_1}\cdots \gamma_n^{a_n}$ is an interval in $\{1,\ldots,n\}$, we answer this question by computing these degrees via the Tutte polynomial evaluation $T_M(1,y)$ and the \emph{mixed Eulerian numbers} $A_{b_1,\ldots,b_{r}}$ of Postnikov
\cite[Section~16]{Postnikov}.
\begin{thm}
\label{thm:mixedintro}
For $a_{1},\ldots,a_k\ge 1$ with $\sum a_i=r$, define the polynomial $A_{a_1,\ldots,a_k}(y)=\sum_{i=0}^{r-k}A_{0^{i},a_{1},\ldots,a_k,0^{r-k-i}}y^i$. Then we have
\begin{align*}T_M(1,y)&=\frac{1}{A_{a_1,\ldots,a_k}(y)}\sum_{i=0}^{n-k} \deg_{A^\bullet(M)}(\gamma_{i+1}^{a_{1}}\ldots \gamma_{i+k}^{a_k}) y^i.\end{align*}
\end{thm}
The significance of $T_M(1,y)$ for the $h$-vector follows from the fact that
$$[y^i]T_{M^*}(1,y)=h_{\rk(M)-i},$$
where $M^*$ is the matroid dual of $M$.
Hence the intersection numbers we are able to compute are directly related to log-concavity statements associated matroid $h$-vector sequences.

In \Cref{thm:mixedEuleriangenerating} we show that
$$\frac{A_{a_1,\ldots,a_k}(y)}{(1-y)^{r+1}}=\sum_{i=0}^{\infty}(i+1)^{a_1}\ldots (i+k)^{a_k}y^i,$$
generalizing the known identity for the Eulerian polynomial $A_r(y)$, which is the specialization to $k=1$.\footnote{Nadeau and Tewari independently observed this identity, later appearing in \cite{NadeauTewari}.} We also note that the Eulerian polynomial $A_r(y)$ is the generating function for permutations in $S_r$ with exactly $i$ descents, and
$A_{r-k,k}(y)$ is the generating function for ``refined Eulerian
numbers'', permutations of $S_{r+1}$ with $k$ descents and
last element $i+1$ \cite{Consecutive}. We remark that Postnikov \cite[Theorem 16.4]{Postnikov} showed  $A_{a_1,\ldots,a_k}(1)=r!$ using affine Coxeter groups, resolving a conjecture of Stanley. Because $A_{1^r}(y)=r!$, we obtain the following corollary directly extracting $T_M(1,y)$ coefficients.
\begin{cor} The coefficients of $T_M(1,y)$ can be computed as,
$$[y^i]T_M(1,y)=\frac{1}{r!}\deg_{A^\bullet(M)}(\gamma_{i+1}\ldots \gamma_{i+r}).$$
\end{cor}
Combining this with the known result for $\omega\in A^{r-c}(M)$ that \cite[Lemma 5.1]{HK12}
$$\deg_{A^\bullet(M)}(\gamma_n^i \omega)=\deg_{A^\bullet(\tau^iM)}(\omega),$$
where $\tau^i$ denotes matroid truncation to corank $i$, and the formula
$$[x^i]T_M(1+x,y)= T_{\tau^i M}(1,y)-(y-1)T_{\tau^{i-1}M}(1,y),$$
we conclude that in fact every coefficient of the Tutte polynomial $T_M(x,y)$ can be written as explicit linear combinations of degrees of products of $\gamma_{i}$. 

\begin{rmk}
Prior to this work, there was a consensus among the experts in the field that because $A^\bullet(M)$ is constructed using only the information of the lattice of flats, one could not access invariants of matroids like $T_M(1,y)$ through degree computations in $A^\bullet(M)$ which are sensitive to the presence of parallel elements.

How this is mitigated in the present work is that for two matroids $M,M'$ with the same lattice of flats, the elements $\gamma_i$ will not necessarily induce the same elements in $A^1(M)=A^1(M')$ even if both $M,M'$ have the same ground set $\{0,\ldots,n\}$. The induced classes are ultimately determined by the relative positioning of the normal fans of $P_{i,n+1}$ and the Bergman fans $\Delta_M,\Delta_{M'}$ in $\mathbb{R}^{n+1}/\langle (1,\ldots,1)\rangle$, and the embedding of the Bergman fans depend on the information of which elements lie in which flats.
\end{rmk}

\subsection*{Organization} The structure of our paper is as follows. In \cref{MatroidTutte} we give the needed background on matroids. In  \cref{sec:MWmain,sec:Delta,CombChowSection} we collect the needed algebraic background material on Minkowski weights, the permutohedral fan and the Chow rings of matroids. In \cref{AlphaSec1,AlphaSec2} we introduce the nef divisors $\gamma_1,\dots,\gamma_n$ on the permutohedral fan and describe how products of these decompose nicely in terms of what we call symmetrized Minkowski weights, which freely generate the permutation invariant portion of the space of Minkowski weights on the permutohedral fan. Particular focus is given to the case of what we call one-window symmetrized Minkowski weights.

\Cref{sec:sliding sets} introduces the combinatorics of the sliding sets problem, and in \cref{SlidingSetsGeneral} we show how this tool can be used to compute products with symmetrized Minkowski weights. We put this tool to use in \cref{PhirkSection} where we show how $T_M(1,y)$ arises as a product with various one-window symmetrized Minkowski weights, ultimately proving \cref{thm:mixedintro}. This is used in \cref{LogConcavitySection} to prove our main result, \cref{preintrothm}.

In \Cref{KlyachkoAppendix}, we consider questions motivated by Schubert calculus. Specifically, we evaluate Schubert and Schur polynomials in the differences $\gamma_n, \gamma_{n-1}-\gamma_n,\dots, \gamma_1 - \gamma_2,-\gamma_1$. Results of Klyachko \cite{K85} give alternative, simpler, expressions for these classes. Integrating our previous results with those of Klyachko, we are able to connect the images of these classes in $A^\bullet(M)$ to the reliability polynomial of $M$. Finally, in \cref{RecursionAppendix}  we present a generalized deletion-contraction recursion to compute the degree of an arbitrary symmetrized Minkowski weight in the Chow ring of a matroid. However we do not know how to relate the numbers produced by this recursive process to known matroid invariants, which we believe would would be necessary to give a satisfactory answer to  \cref{quest:whatinvariants}.

\section{Matroids and Tutte Polynomials}
\label{MatroidTutte}
In this section, we recall the necessary combinatorial background. A matroid is a combinatorial generalization of both graphs and vector configurations which encodes a notion of ``independence''. For the perspective emphasized in our work we recommend the surveys of Ardila \cite[Section~7]{Ardila} and Katz \cite{KatzSurvey}.
\subsection{Matroids through their flats}

We mostly focus on the description of matroids in terms of flats. A matroid $M$ on ground set $E$ is determined by a collection of subsets $\mathcal{L}^M$ of $E$ (which unless explicitly stated otherwise we will take to be $E=\{0,1,2,\dots,n\}$) called the \emph{flats of $M$} satisfying the following axioms:
\begin{enumerate}
    \item $E\in \mathcal{L}^M$
    \item If $F,G\in \mathcal{L}^M$ then $F\cap G\in \mathcal{L}^M$
    \item If $F$ is a flat, then the minimal flats $G_1,\cdots G_k$  containing $F$ have the property that $G_1\setminus F ,\ldots G_k\setminus F$ are disjoint and their union is $E \setminus F$.
\end{enumerate}
If there is a $K$-vector space $V$ and vectors $v_0,\ldots,v_n\in V$ there is an associated \emph{representable matroid} $M$ over the field $K$ with  $$\mathcal{L}^M:=\{F\subset \{0,\ldots,n\}:v_j\not\in span(\{v_i\}_{i\in F})\text{ for all }j\not \in F\},$$
so the flats of $M$ correspond to the subspaces of $V$ spanned by subsets of the vectors $v_i$. A special case of representable matroids are \emph{graphic matroids}, which for a graph $G$ with vertices $\{0,\ldots,n\}$ has vectors $e_i-e_j$ for all $ij$ in the edge set of $G$.
For the reader new to matroids it would be instructive to verify all of the facts below (and the axioms above) for representable matroids, which all correspond to basic facts about finite-dimensional vector spaces.

An important example of a matroid is the uniform matroid.
\begin{defn}
Let $U_{n+1}$ denote the matroid whose flats are every subset of $\{0,1,\ldots,n\}$. We denote the flats of $U_{n+1}$ (the set of all subsets of $\{0,\ldots,n\}$) by $\mathcal{L}^{n+1}$.
\end{defn}
The uniform matroid $U_{n+1}$ is representable by any basis of $K^{n+1}$.

For every $A\subset E$ there is a unique minimal flat containing $A$, which in the representable setting corresponds to the subspace spanned by $\{v_i\}_{i\in A}$.
\begin{defn}
We denote by $\overline{A}\in \mathcal{L}^M$ for the unique minimal flat containing $A$.
\end{defn}
An important role in the combinatorics of flats will be played by the chains of flats of $M$. Every maximal chain of flats of $M$ between $\overline{\emptyset}$ and $E$ has the same length.
\begin{defn}
If $\overline{\emptyset} \subsetneq F_1 \subsetneq \cdots \subsetneq F_r\subsetneq E$ is a maximal chain of flats, then the \emph{rank} of $M$ is defined by $\rk(M)=r+1$ and the \emph{corank} of $M$ is defined by $\crk(M)=|E|-(r+1)$.
\end{defn}
In the representable case, $\rk(M)$ is the dimension of the span of the vectors $v_0,\ldots,v_n$.
Unless otherwise stated, we will write
$$r+1=\rk(M).$$
\begin{defn}
For $0\le c \le r$ we define $$\mathcal{L}^{M}_{(c)}=\{\mathcal{F}=\{\overline{\emptyset} \ne F_1\subsetneq \ldots \subsetneq F_c \ne E\}: F_i\in \mathcal{L}^M\}.$$
For $\mathcal{F}\in \mathcal{L}^{M}_{(c)}$ we will denote the flats in this chain by roman script $F_1\subsetneq \ldots \subsetneq F_c$ with the convention that $F_0=\overline{\emptyset}$ and $F_{c+1}=E$.
\end{defn}
By convention, $\mathcal{L}^M_{(0)}=\{\emptyset\}$ containing the single empty chain of flats, and since $\mathcal{L}^M_{(1)}=\{\{F\}:F\in \mathcal{L}^M\setminus \{\overline{\emptyset},E\}\}$, we often directly identify
$\mathcal{L}^M_{(1)}=\mathcal{L}^M\setminus \{\overline{\emptyset},E\}$.
 
\begin{defn}
For any subset $A\subset E$, we define $\rk_M(A)$ to be the maximal $c$ such that there is a chain of flats $\mathcal{F}\in \mathcal{L}^M_{(c)}$ with $F_c=\overline{A}$, and we define $\crk_M(A)=|E|-\rk_M(A)$. In particular, $\rk(M):=\rk_M(E)$ and the corank $\crk(M):=\crk_M(E)$ of the matroid $M$.
\end{defn}
In the representable case, $\rk_M(A)$ is the dimension of the span of the vectors lying in $A$.
\begin{rmk}
A matroid can be equivalently axiomatized as a natural-number valued rank function $\rk_M$ on subsets of $E$, which has to satisfy $\rk_M(\emptyset)=0$, $\rk_M(\{i\})\le 1$ for all $i\in E$, and the submodularity condition $\rk_M(A\cup B)+\rk_M(A\cap B)\le \rk_M(A)+\rk_M(B)$ for all $A,B\subset E$.
\end{rmk}

 \begin{defn}
 Given a subset $A\subset E$, we may consider the \emph{restriction} $M|_A$ with $$\mathcal{L}^{M|_{A}}=\{F\cap A:F\in \mathcal{L}^M\},$$ the \emph{deletion} $M\setminus A=M|_{E\setminus A},$ and the \emph{contraction} $M/A$ with
 $$\mathcal{L}^{M/A}:=\{F\setminus A: A\subset F\in \mathcal{L}^M\}.$$
 \end{defn}
 In the representable case, restriction $M|_A$ is the matroid for the vectors $v_i$ with $i\in A$, the deletion $M\setminus A$ is the matroid for the vectors $v_i$ with $i\not\in A$, and the contraction $M/A$ is the matroid for the vectors $\pi(v_i)$ with $i\not\in A$, where $\pi$ is the quotient map $V\to V/span(\{v_i\}_{i\in A})$.

\begin{defn}
An \emph{independent set} $I\subset M$ is a subset such that $\rk_M(I)=|I|$. A \emph{basis} is an independent set of size $\rk_M(M)$. The \emph{independence complex}, $\operatorname{IN}(M)$, of $M$ is the collection of independent sets of $M$.
\end{defn}
It is a basic fact of matroid theory that every independent set is contained in a basis, and every subset of an independent set is independent. It follows that $\operatorname{IN}(M)$ is a pure simplicial complex.
In the representable case, an independent set corresponds to a linearly independent collection of vectors, and a basis corresponds to a basis of $span(\{v_i\}_{i\in E})$. 

\begin{defn}
A \emph{loop} in $M$ is an element of  $\overline{\emptyset}$. A matroid is said to be \emph{loopless} if it has no loops, i.e. $\emptyset=\overline{\emptyset}$. A \emph{coloop} $i\in E$ is an element with $\rk_M(E\setminus i)=\rk_M(M)-1$.
\end{defn}
In the representable case, a loop corresponds to a zero vector, and a coloop corresponds to a vector not contained in the span of the remaining vectors.

\subsection{Tutte Polynomials}
\label{TuttePolySection}
The Tutte polynomial is perhaps the most beloved enumerative invariants of matroids, and the subject of numerous papers (see the references in \cite[Section 7.6]{Ardila}). Given a matroid $M$ on ground set $E$, its Tutte polynomial is recursively determined (in fact, wildly over-determined) by the following deletion-contraction recurrence-relation: we have $T_{\emptyset}(x,y)=1$, and for $i\in E$ we have
$$T_M(x,y)=\begin{cases}T_{M/i}(x,y)+T_{M\setminus i}(x,y)&i\text{ not a loop and not a coloop,}\\
yT_{M\setminus i}(x,y)&i\text{ a loop,}\\
xT_{M/i}(x,y)&i\text{ a coloop.}\end{cases}.$$
There is a unique polynomial satisfying this recurrence and it is immediate that its coefficients are non-negative integers (see, e.g., \cite{Ardila}).

Using the deletion-contraction relation, one can easily show that
\[
T_M(x+1,1) = \sum_k \#\{I \in \operatorname{IN}(M) : |I| = \operatorname{rk}(M) - k \} x^k.
\]
The usual definition of the $h$-vector of a simplicial complex comes from replacing $x$ with $x-1$ in the generating function for faces  by their codimension. We thus have the following.
\begin{defn}
The \emph{$h$-vector} $h_0,\dots,h_{\operatorname{rk}(M)}$ of $M$ is the coefficient sequence of the $h$-polynomial $T_M(x,1) = \sum_{k=0}^{\operatorname{rk}(M)} h_{\operatorname{rk}(M) -k}x^k$.
\end{defn}
There is a well-known relationship between the Tutte polynomial of a matroid $M$ and its dual $M^*$ (whose bases are complements of bases of $M$): $T_M(x,y) = T_{M^*}(y,x)$. This gives the following fact.
\begin{prop}\label{fact:Dawsonfact}
 For a matroid $M$ with $h$-vector $h_0,\dots,h_{\operatorname{rk}(M)}$ and dual matroid $M^*$, we have $T_{M^*}(1,y)=\sum_{i=0}^{\rk(M)}h_{\rk(M)-i}y^i$.
\end{prop}
This fact is useful since in our later results $T_M(1,y)$ appears naturally.
\begin{cor}
\label{cor:nointernal}
The $h$-vector of a matroid is a non-negative integer sequence with no internal zeros.
\end{cor}
\begin{proof}
Non-negativity follows from the non-negativity of the Tutte polynomial. The deletion-contraction recurrence recursively shows that the set of $i$ such that $[x^i]T_M(x,1)\ne 0$ is a discrete interval of the form $\{a,\ldots,\rk(M)\}$, which verifies the no internal zeros property.
\end{proof}

Finally, we record an explicit formula for the coefficients of the Tutte polynomial in terms of basis activity that we will need later.
\begin{defn}
Let $B$ be a basis of $M$. An element $j \notin B$ is  \emph{externally active} for $B$ if for all $i<j$, the set $B-i \cup j$ is \textit{not} a basis of $M$. Write $\operatorname{ex}(B)$ for the set of  externally active elements for $B$.

Dually, an element $i \in B$ is  \emph{internally active} for $B$ if for all $j<i$, $B-i \cup j$ is \textit{not} a basis of $M$. Write $\operatorname{in}(B)$ for the set of internally active elements of $B$.
\end{defn}

\begin{prop}[Tutte~\cite{Tutte}]\label{inexTutte} For all $i,j$, the coefficient of $x^iy^j$ in $T_M(x,y)$ is the number of bases $B$ of $M$ with $|\operatorname{in}(B)|=i$ and $|\operatorname{ex}(B)|=j$.
In particular,
  $$[y^j]T_M(1,y)=|\{B\text{ a basis of $M$}:|\operatorname{ex}(B)|=j|\}.$$
\end{prop}

\section{Minkowski weights}\label{sec:MWmain}
In this section we describe the ring of Minkowski weights on a complete rational fan, following Fulton and Sturmfels \cite{FS97}, which is canonically isomorphic to the operational Chow ring of the associated complete toric variety. 
\label{ToricBackground}

\subsection{Fans}
Let $N$ be a lattice and let $N_{\mathbb{R}}:=N \otimes \mathbb{R}$. Let $\Delta$ be a rational fan in $N_{\mathbb{R}}$ -- that is a fan whose cones are generated by elements of $N$. We assume all fans are rational.

A fan $\Delta$ is \emph{complete} if the union of its cones is all of $N_{\mathbb{R}}$. We say that $\Delta$ is \emph{unimodular} or \emph{smooth} if every maximal cone (equiv. every cone) of $\Delta$ is generated by part of a basis of the lattice $N$.

The set of $i$-dimensional cones of $\Delta$ is denoted $\Delta_{(i)}$. In particular, $\Delta_{(\dim \Delta)}$ are the maximal cones and $\Delta_{(0)}=\{(0)\}$. For $\sigma\in \Delta$, we let $N_{\sigma}$ be the sub-lattice of $N$ spanned by $N\cap \sigma$.

\subsection{Minkowski weights}
\label{sec:MW}
In this subsection we take $\Delta$ to be a complete fan. We begin by defining the group of Minkowski weights on $\Delta$. 
\begin{defn}
A codimension $k$ \emph{Minkowski weight} on $\Delta$ is a function $\Phi:\Delta_{(\dim \Delta-k)}\to \mathbb{Z}$ satisfying the \emph{balancing condition}: For all $\tau \in \Delta_{(\dim \Delta-(k+1))}$ we have,
\[\sum_{\tau\subset \sigma \in \Delta_{(\dim\Delta-k)}} \Phi(\sigma)v_{\sigma/\tau}=0\in N/N_{\tau}\]
where $v_{\sigma/\tau}\in N/N_{\tau}$ is the primitive generator of the ray $(\sigma+N_{\tau})/N_{\tau}$.
Write $MW^k(\Delta)$ for the additive group of codimension $k$ Minkowski weights on $\Delta$ and $MW^\bullet(\Delta)=\bigoplus_{k=0}^{\dim \Delta} MW^k(\Delta)$.
\end{defn}
Since there is a unique cone $\{0\}\in \Delta_{(0)}$, we have a \emph{degree} map
$$\deg: MW^{\dim \Delta}(\Delta)\to \mathbb{Z}$$  defined by $\deg(\Phi) = \Phi(\{0\})$. 

There is a natural ring structure on $MW^\bullet(\Delta)$, computed explicitly using the following ``fan displacement rule''.
\begin{defn}[{\cite[Proposition 4.2]{FS97}}]\label{def:fanDisp}
For $\Phi_1\in MW^{c_1}(\Delta)$, $\Phi_2\in MW^{c_2}(\Delta)$, we define the product Minkowski weight $\Phi_1\cdot \Phi_2\in MW^{c_1+c_2}(\Delta)$ as follows.

Choose a generic vector $v\in N_{\mathbb{R}}$ (outside of a finite collection of hyperplanes), and for $\Phi_1\in MW^{c_1}(\Delta)$, $\Phi_2\in MW^{c_2}(\Delta)$, $\gamma \in \Delta_{(\dim\Delta-(c_1+c_2))}$ define
$$(\Phi_1 \cdot \Phi_2)(\gamma)=\sum_{(\sigma_1,\sigma_2)\in \Delta_{(\dim \Delta-c_1)}\times \Delta_{(\dim \Delta-c_2)}} m^{\gamma,v}_{\sigma_1,\sigma_2}\Phi_1(\sigma_1)\Phi_2(\sigma_2),$$
where
$$m^{\gamma,v}_{\sigma_1,\sigma_2}=\begin{cases}[N:N_{\sigma_1}+N_{\sigma_2}]&\sigma_1 \cap (\sigma_2+v) \ne \emptyset \text{ and } \gamma\subset \sigma_1,\sigma_2,\\0&\text{otherwise.}\end{cases}$$
Note that the condition $\gamma \subset \sigma_1,\sigma_2$ ensures that $N_{\sigma_1}+N_{\sigma_2}$ is a full-dimensional sublattice of $N$, and hence of finite index. The multiplicities $m^{\gamma,v}_{\sigma_1,\sigma_2}$ depend on $v$, but the formula for the product $(\Phi_1 \cdot \Phi_2)(\gamma)$ does not. 
\end{defn}
\begin{rmk}
The ring $MW^\bullet(\Delta)$ is isomorphic to the operational Chow ring $A^{\bullet}(X(\Delta))$ of the associated toric variety $X(\Delta)$ \cite[Theorem 3.1]{FS97}. The identification uses the fact that the Kronecker duality map 
\begin{align*}
    A^{k}(X(\Delta))\to \operatorname{Hom}(A_k(X(\Delta)), \mathbb{Z})
\end{align*}
is an isomorphism for complete toric varieties \cite[Theorem 3]{FMSS95}. The group $A_k(X(\Delta))$ is spanned by the $k$-dimensional toric subvarieties of $X(\Delta)$, which are indexed by the codimension $k$ cones of $\Delta$, and relations between the subvarieties are exactly encoded by the balancing conditions in $MW^k(\Delta)$. See also \cref{thm:MWvsA} below. 
\end{rmk}

%
\subsection{Minkowski Weights associated to polytopes}
\label{sec:polytope}
Here we record how a polytope gives rise to a Minkowski weight, and its relation to volumes and mixed volumes. We assume that $\Delta$ is a complete unimodular fan throughout. A reference for the material here is \cite[Theorem 5.1]{FS97}, the proof of which directly connects Minkowski weights and volumes of faces of polytopes as we describe, and \cite[Section 5.4]{F93} for mixed volumes. Let $M=N^{\vee}$ be the dual lattice to $N$, and $M_{\mathbb{R}}:=M\otimes \mathbb{R}$. Given a full dimensional polytope $P\subset M_{\mathbb{R}}$, its normal fan $\Delta_P\subset N_{\mathbb{R}}$ consists of the cones 
\begin{align*}
    \sigma_{Q}:=\{u\in N_{\mathbb{R}}\mid \langle u,v\rangle = \max\{\langle u,p\rangle\mid p\in P\}\text{ for all }v\in Q\},
\end{align*}
as $Q$ varies over all the faces of $P$. If $P$ is integral in $M$ and $\Delta_P$ refines a complete unimodular fan $\Delta$, then there is an associated Minkowski weight $\alpha\in MW^1(\Delta)$ which has the property that for $1\le i \le \dim(N_{\mathbb{R}})$, the Minkowski weight $\alpha^i\in MW^i(\Delta)$ has $\alpha^i(C)=0$ if $C\in \Delta_{(\dim \Delta-i)}$ is not contained in a cone $\sigma_{Q}$ with $Q$ a face of dimension $i$, and otherwise
$$\alpha^i(C)= \dim(Q)!\operatorname{vol}(Q),$$
where volume is normalized so that the fundamental parallelepiped of the intersection of $N$ with the affine hull of $Q$ has volume $1$. In particular,
\begin{align*}
    \deg(\alpha^{\dim(N)})=\dim(N)!\operatorname{vol}(P).
\end{align*}
Generalizing this last equality, if $\alpha_1,\ldots,\alpha_{\dim(N)}\in MW^1(\Delta)$ are associated to polytopes $P_1,\ldots,P_{\dim(N)}$, then
\begin{align*}
    \deg(\alpha_1\cdots\alpha_{\dim(N)})=\dim(N)!\operatorname{MV}(P_1,\ldots,P_{\dim(N)}),
\end{align*}
where $\operatorname{MV}(P_1,\ldots,P_{\dim(N)})$ denotes the \emph{mixed volume} normalized to $M$ of $P_1,\ldots,P_{\dim(N)}$.

Write $$MW^\bullet(\Delta)_{\mathbb{R}}:=MW^\bullet(\Delta)\otimes \mathbb{R}$$
for the ring of real Minkowski weights, with degree map $\deg:MW^{\dim \Delta}(\Delta)\to \mathbb{R}$.
If $P$ (or $P_1,\ldots,P_{\dim(N)}$) have normal fan(s) refining $\Delta$ but are not necessarily integral, then they induce in the same way real Minkowski weights in $MW^1(\Delta)_{\mathbb{R}}$ satisfying the same formulas, where volume in the first case is normalized to a (nonempty) lattice obtained by intersecting a translate of $M$ with the affine hull of $Q$.

Mixed volumes of polytopes satisfy a log-concavity inequality called the Aleksandroff-Fenchel inequality. As a point of comparison to \Cref{thm:MinkAHK}, we state this inequality using the degree formalism.
\begin{thm}
If $L_1,\ldots,L_{n-2}.H,K\in MW^1(\Delta)_{\mathbb{R}}$ are Minkowski weights associated to polytopes, then if $\Omega=L_1\cdots L_{n-2}$, the sequence of numbers
$$\deg(\Omega\cdot H^2),\deg(\Omega \cdot HK), \deg(\Omega\cdot K^2)$$
is a log-concave sequence.
\end{thm}
\section{The permutohedral fan $\Delta_{n+1}$}\label{sec:Delta}
Most of our computations will take place in the ring of Minkowski weights on the permutohedral fan, which we introduce here. This fan is universal, in a sense, for doing computations that will descend to the Chow rings of all matroids on a common ground set (see \cref{thm:bes} below). 

Let $N=\mathbb{Z}^{n+1}/\langle (1,\ldots 1)\rangle$ with dual lattice $M=\{(x_0,\ldots,x_n)\in \mathbb{Z}^{n+1}:\sum x_i=0\}$. Write $e_0,\dots,e_n$ for the images of standard basis of $\mathbb{Z}^{n+1}$ in $N$ and $N_{\mathbb{R}}$, and let $e_F$ denote the sum $\sum_{i \in F} e_i$ for a proper non-empty subset $F$ of $\{0,\dots,n\}$.
\begin{defn}
The $n$-dimensional permutohedron is the convex hull
$$\operatorname{Perm}_n=\operatorname{conv}\{(\sigma(0),\ldots,\sigma(n)\}:\sigma\text{ a permutation of }\{0,\cdots,n\}\}.$$
The permutohedral fan $\Delta_{n+1}\subset N_{\mathbb{R}}$ is the normal fan of $\operatorname{Perm}_n-v\subset M_{\mathbb{R}}$, for any vector $v\in \{\sum x_i=\frac{n(n+1)}{2}\}\cap \mathbb{Z}^{n+1}$.
\end{defn}
The permutohedral fan is also known to be the fan induced by the type $A_n$ braid arrangement, which is image in $N_{\mathbb{R}}$ of the collection of hyperplanes in $\mathbb{R}^{n+1}$ where two coordinates are equal.
\begin{defn}
For $\mathcal{F}=\{\emptyset \subsetneq F_1\subsetneq \cdots \subsetneq F_c\subsetneq \{0,\ldots,n+1\}\}\in \mathcal{L}^{n+1}_{(c)}$, define the cone
$$\sigma_{\mathcal{F}}=\sum_{i=0}^{c-1} \mathbb{R}_{\ge 0}e_{F_{i+1}}=\{\sum_{i=0}^{c} t_ie_{F_{i+1}\setminus F_{i}}:t_0\ge t_1 \ge \ldots \ge t_{c-1} \ge t_c=0\} \subset N_{\mathbb{R}}.$$
\end{defn}
Note that $\sigma_{\mathcal{F}}$ is a unimodular cone. The following result is well-known.
\begin{prop}
The $c$-dimensional cones of $\Delta_{n+1}$ are given by
$$\Delta_{n+1,(c)}=\{\sigma_{\mathcal{F}}:\mathcal{F}\in \mathcal{L}^{n+1}_{(c)}\}.$$
In particular, $\Delta_{n+1}$ is a complete unimodular fan.
\end{prop}

We conclude by proving that when computing products of Minkowski weights in $MW^\bullet(\Delta_{n+1})$ using the fan displacement rule of \cref{def:fanDisp}, the non-zero multiplicities $m^{\gamma,v}_{\sigma_1,\sigma_2}$ therein are all equal to $1$.
\begin{lem}
\label{lem:mult1}
Let $\mathcal{F}\in \mathcal{L}^{n+1}_{(k_1)}$ and $\mathcal{G}\in \mathcal{L}^{n+1}_{(k_2)}$. Then the coefficient $m^{\gamma,v}_{\sigma_\mathcal{F},\sigma_\mathcal{G}}$ equals $1$ precisely when $\sigma_\mathcal{F} \cap (\sigma_\mathcal{G}+v) \ne \emptyset$ and $\gamma \subset \sigma_\mathcal{F},\sigma_\mathcal{G}$, and is $0$ otherwise.
\end{lem}
\begin{proof}
We have $N_{\sigma}=\langle e_{P_0},\ldots,e_{P_{k_1}}\rangle$ where $P_i=F_{i+1}\setminus F_i$ (setting $F_0=\emptyset$ and $F_{k_1}=\{0,\ldots,n\}$), and similarly $N_\tau=\langle e_{Q_1},\ldots,e_{Q_{k_2+1}}\rangle$ where $Q_i=G_{i+1}\setminus G_i$, which are two partitions
$$\{0,\ldots,n\}=\bigsqcup_{i=0}^{k_1}P_i=\bigsqcup_{i=0}^{k_2}Q_i.$$ Assuming $m^{\gamma,v}_{\sigma_1,\sigma_2}\ne 0$ we have $m^{\gamma,v}_{\sigma_1,\sigma_2}=[N:N_{\sigma}+N_{\tau}]$ is finite. We will show that this index being finite implies that it is $1$. This follows from the lemma below.
\end{proof}

\begin{lem}
Suppose we have partitions $\{0,\ldots,n\}=\bigsqcup_{i=1}^{\ell_1}P_i=\bigsqcup_{i=1}^{\ell_2}Q_i$, such that the lattice $\langle e_{P_1},\ldots,e_{P_{\ell_1}},e_{Q_{1}},\ldots,e_{Q_{\ell_2}}\rangle$ is of finite-index inside $\mathbb{Z}^{n+1}/\langle e_0+\ldots+e_n\rangle$. Then this index is $1$.\end{lem}
\begin{proof}
Let $\Lambda_P=\langle e_{P_1},\ldots,e_{P_{\ell_1}} \rangle$ and $\Lambda_Q=\langle e_{Q_1},\ldots,e_{Q_{\ell_2}}\rangle$.
We proceed by induction on $n$.
If any of the $P_i$ or $Q_j$ are singletons, without loss of generality $P_1=\{n\}\subset Q_1$, then $\Lambda_P+\Lambda_Q$ contains $\langle e_n \rangle$, and the result follows from the inductive hypothesis applied to $P=\{P_2,\ldots, P_{\ell_1}\}$ and $Q=\{Q_1\setminus n,Q_2,\ldots, Q_{\ell_2}\}$ in $\mathbb{Z}^n/\langle e_0+\ldots+e_{n-1}\rangle$. Hence assume now that all $P_i,Q_j$ have size at least $2$. We will show that $\langle e_{P_1},\ldots,e_{P_{\ell_1}},e_{Q_{1}},\ldots,e_{Q_{\ell_2}}\rangle$ is not finite-index inside $\mathbb{Z}^{n+1}/\langle e_0+\ldots+e_n\rangle$. 

Because every index $x\in \{0,\ldots,n\}$ belongs to exactly one $P_i$ and one $Q_j$, we can find distinct elements
$i_0,\ldots,i_{m-1},i_{m}=i_{0}$ such that $i_{h},i_{h+1}$ are equivalent under the partition $P$ if $h$ is even and $i_{j},i_{j+1}$ are equivalent under the partition $Q$ if $j$ is odd. We may assume that $m$ is even, as if $m$ is odd then $i_{m-1},i_m=i_0,i_1$ are all equivalent under $P$, so we may go directly from $i_{m-1}$ to $i_1$ in the cycle, skipping $i_m$. Then writing $m=2k$ we have
\begin{align*}\Lambda_P &\subset \{x_{i_{0}}-x_{i_{1}}=0\}\cap \{x_{i_{2}}-x_{i_{3}}=0\}\cap \ldots \cap \{x_{i_{2k-2}}-x_{i_{2k-1}}=0\},\text{ and}\\
\Lambda_Q&\subset \{x_{i_{1}}-x_{i_{2}}=0\}\cap \{x_{i_{3}}-x_{i_{4}}=0\}\cap \ldots \cap \{x_{i_{2k-1}}-x_{i_{0}}=0\}\text{, so}\end{align*}
$$\Lambda_P+\Lambda_Q\subset \{x_{i_0}-x_{i_1}+x_{i_2}-\ldots-x_{i_{2\ell-1}}=0\},$$
and hence cannot be of finite index in $\mathbb{Z}^n/\langle e_0+\ldots+e_n\rangle$, a contradiction.
\end{proof}
\section{Combinatorial Chow rings}
\label{CombChowSection}
In this section we define the Chow ring of a matroid $M$. We describe how to relate computations done in the ring of Minkowski weights on $\Delta_{n+1}$ to computations done in the Chow ring of $M$. This allows us to use the powerful results of \cite{AHK18} in the setting of Minkowski weights. We will always assume from now on the $M$ is a loopless rank $r+1$ matroid on $\{0,\dots,n\}$. It is important to note that we do allow $M$ to have parallel elements (rank one flats of cardinality larger than one).
\subsection{Chow rings of matroids}
We define the combinatorial Chow ring following \cite{FY04,AHK18}.
\begin{defn}
The \emph{combinatorial Chow ring} of $M$ is
\[
A^{\bullet}(M)=\mathbb{Z}\left[x_F: F\in \mathcal{L}^{M}_{(1)}\right]/\left(I_M + J_M\right)
\]
where the ideals $
I_M = \langle x_F x_G : F \textup{ and } G \textup{ are incomparable in }\mathcal{L}^{M}_{(1)} \rangle,$
and
$
J_M = \langle \sum_{i \in F} x_F - \sum_{j \in G} x_G : 0 \leq i < j \leq n\rangle.
$ For $\mathcal{F}\in \mathcal{L}^M_{(c)}$ write $x_{\mathcal{F}}=\prod_{i=1}^c x_{F_i}$.
\end{defn}
This is a $\mathbb{Z}$-graded ring, generated in degree $1$ by the $x_F$. We have $A^0(M)= \mathbb{Z}$, and  $A^{m}(M)=0$ for any $m>r$. Every product $x_{\mathcal{F}}$ with $\mathcal{F}\in \mathcal{L}^{M}_{(r)}$ is equal to the same generator of $A^r(\Delta_M)$ \cite[Proposition~5.8]{AHK18}, and sending this product to $1\in \mathbb{Z}$ induces a natural isomorphism
$$\deg:A^r(M)\to \mathbb{Z}$$
called the degree.

\begin{defn}
We set $A^\bullet(M)_{\mathbb{R}}:=A^\bullet(M)\otimes \mathbb{R}$, and define the \emph{nef cone} $A^1_{nef}(M)\subset A^1(M)_{\mathbb{R}}$ to be the real cone generated by sums $\sum c(u_F)x_F\in A^1(M)_{\mathbb{R}}$ where $c$ is a convex homogeneous piecewise-linear function on the fan $\Delta_{n+1}$. We call an element of the nef cone a \emph{nef divisor}.
\end{defn}
\subsection{Bergman fans, Minkowski weights, and $A^\bullet(U_{n+1})$}
There is a fan $\Delta_M$ canonically associated to $M$ which encodes chains of flats in $M$.
\begin{defn}
The \emph{Bergman fan} $\Delta_M$ of $M$ is the $r$-dimensional fan in $N=\mathbb{Z}^{n+1}/\langle e_0+\cdots+e_n\rangle$ with $c$-dimensional cones  $\Delta_{M,(c)}=\{\sigma_{\mathcal{F}}:\mathcal{F}\in \mathcal{L}^{M}_{(c)}\}$ for $0\le c \le r$.
\end{defn}
The Bergman fan of a matroid is unimodular \cite[Proposition~2.4]{AHK18}, but almost never complete. However, for $M=U_{n+1}$ we have $\Delta_{U_{n+1}}=\Delta_{n+1}$, which is a complete unimodular fan. Because of this, the combinatorial Chow ring for $U_{n+1}$ can be identified with the ring of Minkowski weights on $\Delta_{n+1}$.
\begin{thm}\label{thm:MWvsA}
The map
\[A^\bullet(U_{n+1}) \to MW^{\bullet}(\Delta_{n+1}),\quad
z \mapsto (\sigma_{\mathcal{F}} \mapsto \deg(zx_\mathcal{F})),
\] 
is an isomorphism of rings. Under this map, the degree of an element $A^{n}(\Delta_{n+1})$ equals the degree of the corresponding Minkowski weight. Further, the Minkowski weight  corresponding to a polytope $P$ is identified with $\sum h_P(u_F)x_F\in A^1_{nef}(\Delta_{n+1})$, where $h_P$ is the convex homogenous piecewise-linear  support function $$h_P(v)=\max_{p\in P} \langle v,p\rangle.$$
\end{thm}
\begin{proof}
The isomorphism of rings follows from the main result of \cite{FS97}, since the ring $A^\bullet(U_{n+1})$ is the Chow ring of the unimodular complete fan $\Delta_{n+1}$. Here one must use the fact that the Chow ring used in \cite{FS97} is the \textit{operational} Chow ring, and for a unimodular complete fan this is canonically isomorphic to the Chow homology group endowed with the intersection product \cite[Corollary~17.4]{F93}. The statement identifying the nef classes of polytopes follows from \cite[Theorem~5.1]{FS97} since $\sum h_P(e_F) x_F$ is the class of ample divisor associated to the polytope $P$.
\end{proof}
Even though there is not a ring of Minkowski weights associated to an arbitrary matroid $M$, we can compute degrees in $A^\bullet(M)$ using Minkowski weights on $\Delta_{n+1}$. In this way, we can obtain log-concavity statements exclusively from Minkowski weight computations.

\begin{defn}
\label{defn:BergmanMinkowski}
The \emph{Bergman fan Minkowski weight} $[\Delta_M]\in MW^{n-r}(\Delta_{n+1})$ is defined by setting for $\mathcal{F}\in \mathcal{L}^{n+1}_{(r)}$
$$[\Delta_M](\sigma_{\mathcal{F}})=\begin{cases}1 & \mathcal{F}\in \mathcal{L}^M_{(r)}\\0 &\text{otherwise.}\end{cases}$$
Abusing notation, we will identify $[\Delta_M] \in MW^{n-r}(\Delta_{n+1})$ with its image in $A^{n-r}(\Delta_{n+1})$ under the isomorphism from \cref{thm:MWvsA}.
\end{defn}
A proof that $[\Delta_M]$ satisfies the balancing condition can be found in \cite[Theorem~4.2.6]{MS15}. The next result allows us to do many computations in the Chow ring of the permutohedral fan.
\begin{thm}\label{thm:bes}
There is a surjective map
$$\pi_M:A^\bullet(U_{n+1})\to A^\bullet(M)$$ defined for $S\in \mathcal{L}^M_{(1)}$ by
$$x_S\mapsto \begin{cases}x_S & \text{$S\in \mathcal{L}_{(1)}^M$}\\0&\text{otherwise.}\end{cases}$$
The kernel of this map is the annihilator of $[\Delta_M]$.
There is an induced map $A^1_{nef}(U_{n+1})\to A^1_{nef}(M)$. 
For $\Omega\in A^{n-r}(U_{n+1})$ we have $$\deg_{A^\bullet(U_{n+1})}([\Delta_M]\cdot \Omega)=\deg_{A^\bullet(M)}(\pi_M(\Omega)).$$
\end{thm}
\begin{proof}
The first statement follows at once from the definition of the Chow ring of $M$. That the kernel is the annihilator occurs as \cite[Theorem 4.2.1]{BES19}. The statement relating the nef cones is \cite[Proposition~4.4]{AHK18}. Finally, the statement on the degrees follows from the isomorphisms 
$$A^r(M) \leftarrow A^r(U_{n+1})/\operatorname{ann}([\Delta_M]) \stackrel{-\cdot [\Delta_M]}{\to} A^n(U_{n+1}),$$
provided we can show that an element of degree $1$ on the left corresponds  to an element of degree $1$ on the right. For this we note that for any $\mathcal{F}\in \mathcal{L}^M_{(r)}$, $x_{\mathcal{F}}$ has degree $1$ on the left, and this corresponds to the product $x_{\mathcal{F}} [\Delta_M]$ on the right. By the isomorphism from \cref{thm:MWvsA}, this is exactly the Bergman fan Minkowski weight $[\Delta_M] \in MW^r(\Delta_{n+1})$ evaluated on $\sigma_{\mathcal{F}}$, which is $1$ by definition of the Bergman fan Minkowski weight.
\end{proof}

Combining this with the isomorphism $MW^\bullet(\Delta_{n+1})\cong A^\bullet(\Delta_{n+1})$, we can rephrase the Mixed Hodge-Riemann relations of \cite[Theorem~8.9]{AHK18} in terms of Minkowski weight intersections as follows.
\begin{thm}
\label{thm:MinkAHK}
If $L_1,\ldots,L_{n-r-2},H,K\in MW^1(\Delta_{n+1})_{\mathbb{R}}$ are the Minkowski weights of generalized permutahedra, then denoting $\Omega=L_1\cdots L_{n-r-2}$, the sequence of numbers
$$\deg([\Delta_M]\cdot \Omega \cdot H^2),\deg([\Delta_M]\cdot \Omega \cdot HK),\deg([\Delta_M]\cdot \Omega \cdot K^2)$$
is a log-concave sequence of nonnegative real numbers.
\end{thm}
\begin{rmk}
The Hodge-Riemann relations established in \cite{ADH20} imply a similar theorem with $\Delta_M$  replaced with an arbitrary product $\Delta_{M_1}\times \cdots \times \Delta_{M_k}$, and $\Delta_{n+1}$ replaced with any fan $\Delta$ refining this product.
\end{rmk}

\subsection{Restriction maps}
In this subsection we describe a family of restriction maps that we will use later. For a non-empty subset $F \subset \{0,\dots, n\}$, we write $U_F$ for the uniform matroid on ground set $F$ and $\Delta_F$ for the permutohedral fan on the index set $F$. That is, in the construction of the permutohedral fan, replace $\mathbb{R}^{n+1}$, $\mathbb{Z}^{n+1}$, etc., with $\mathbb{R}^F$, $\mathbb{Z}^F$, etc.
\begin{prop}
There is a restriction map $\operatorname{res}_F:A^\bullet(\Delta_{n+1}) \to A^\bullet(\Delta_{F} ) \otimes  A^\bullet(\Delta_{\{0,\dots,n\} \setminus F} )$ defined on the generators of $A^\bullet(\Delta_{n+1})$ by the following rule:
\[
x_G \mapsto
\begin{cases}
x_G \otimes 1 & \textup{if } G \subsetneq F, \\
1 \otimes x_{G \setminus F} & \textup{if } F \subsetneq G,\\
0 &\textup{otherwise}.
\end{cases}
\]
If $\alpha \in A^{n-1}(\Delta_{n+1})$ then $\deg( \alpha x_F ) = \deg( \operatorname{res}_F(\alpha))$, where degree on the right side means product of the degrees of the tensor factors.
\end{prop}
\begin{proof}
The description of $\operatorname{res}_F$ follows from the description of the pullback morphism in \cite[Section 5]{AHK18} and then applying \cite[Proposition 3.5(1)]{AHK18}. The second claim follows by verifying the result is true on a generating subset of $A^{n-1}(\Delta_{n+1})$. Since $A^{n-1}(\Delta_{n+1})$ is generated by products of the form $x_{\mathcal{G}}$, where $\mathcal{G}\in \mathcal{L}^{n+1}_{(n-1)}$, the result follows at once from the definition of the degree.
\end{proof}
Repeatedly applying the above result gives the following.
\begin{prop}\label{prop:starprod}
For $\mathcal{F}:\emptyset \subsetneq F_1 \subsetneq \dots \subsetneq F_c \subsetneq \{0,\dots,n\}$, there is a surjective degree preserving ring homomorphism
\[
\operatorname{res}_{\mathcal{F}}:A^\bullet(U_{n+1}) \to 
A^\bullet(U_{F_1}) \otimes  A^\bullet(U_{F_2 \setminus F_1}) \otimes \dots \otimes A^\bullet(U_{\{0,\dots,n\} \setminus F} )
\]
defined on the generators of $A^\bullet(U_{n+1})$ by the rule
\[
x_G \mapsto
\begin{cases}
1^{\otimes (j)} \otimes x_{G \setminus F_{j}} \otimes 1^{\otimes (c-j)} & \textup{if } F_{j} \subsetneq G \subsetneq F_{j+1}, \\
0 &\textup{otherwise},
\end{cases}
\]
where we tacitly define $F_0 = \emptyset$ and $F_{c+1} = \{0,\dots,n\}$. If $\omega \in A^{n-c}(U_{n+1})$ then $$\deg( \omega x_{F_1} \dots x_{F_c} ) = \deg(\operatorname{res}_{\mathcal{F}}(\omega)),$$
where degree on the right side means product of the degrees of the tensor factors.
\end{prop}
Due to the isomorphism $A^\bullet(U_{n+1}) \to MW^\bullet(\Delta_{n+1})$ from \cref{thm:MWvsA} we have a corresponding map
\[
\operatorname{res}_{\mathcal{F}} : MW^\bullet(\Delta_{n+1}) \to 
MW^\bullet(\Delta_{F_1}) \otimes  MW^\bullet(\Delta_{F_2 \setminus F_1}) \otimes \dots \otimes MW^\bullet(\Delta_{\{0,\dots,n\} \setminus F} ).
\]
The target of this map has a degree map, which is the product of the degrees of the tensor factors. We have the following useful result.
\begin{prop}
\label{starprop}
If $\Phi\in MW^{k}(\Delta_{n+1})$, $\mathcal{F}\in \mathcal{L}^{n+1}_{(c)}$, and $\mathcal{G}_i\in \mathcal{L}^{F_{i+1}\setminus F_i}_{(g_i)}$ for $i=0,\ldots,c$ with $\sum g_i=n-k-c$, then 
$$\operatorname{res}_{\mathcal{F}}(\Phi)(\sigma_{\mathcal{G}_0}\otimes \cdots \otimes \sigma_{\mathcal{G}_{c}})=\Phi(\sigma_{\mathcal{H}})$$
where $\mathcal{H}=\mathcal{F}\cup \bigcup_{i=0}^{c} \{G\cup F_i:G\in\mathcal{G}_i\}$.
\end{prop}
\begin{proof}
We have the chain of equalities
\begin{align*}\Phi(\sigma_{\mathcal{H}})=&\deg\left(\Phi\cdot x_{\mathcal{F}} \prod_{i=0}^c\prod_{S\in \{G\cup F_i: G\in \mathcal{G}_i\}} x_{S}\right)\\=&\deg\left(\operatorname{res}_{\mathcal{F}}(\Phi)\cdot \left(\bigotimes_{i=0}^c\prod_{S\in \{G\cup F_i: G\in \mathcal{G}_i\}} x_{S\setminus F_i}\right)\right)\\
=&\deg\left(\operatorname{res}_{\mathcal{F}}(\Phi)\cdot \bigotimes_{i=0}^c x_{\mathcal{G}_i}\right)\\=&\operatorname{res}_{\mathcal{F}}(\Phi)\left(\bigotimes_{i=0}^c \sigma_{\mathcal{G}_i}\right).\qedhere\end{align*}
\end{proof}
\begin{cor}
\label{cor:starcor}
If $\Phi\in MW^c(\Delta_{n+1})$ and $\mathcal{F}\in \mathcal{L}^{n+1}_{(c)}$, then we have
$\Phi(\mathcal{F})=\deg(\operatorname{res}_{\mathcal{F}}(\Phi))$.
\end{cor}
\section{The nef divisors $\gamma_k$}
\label{AlphaSec1}
In this section we define certain classes $\gamma_k\in MW^1(\Delta_{n+1})$ which generalize the classes $\alpha_M$ and $\beta_M$ used by Huh and Katz and later Adiprosito, Huh and Katz \cite{HK12,AHK18} to compute coefficients of the characteristic polynomial of a matroid. The $\gamma_i$ will eventually be related to mixed Eulerian numbers and the Tutte polynomial of a matroid.
\begin{defn}
For $1 \leq k \leq n$ define the hypersimplex $P_{k,n+1}$ to be the convex hull in $\mathbf{R}^{n+1}$ of the set $\{ e_F : F \subset \{0,\ldots,n\}\text{ and } |F| = k\}$.
\end{defn}
It is well known that the vertices of $P_{k,n+1}$ are exactly those $e_F$ with  $|F| = k$.
\begin{defn}
Define $\gamma_k\in MW^1(\Delta_{n+1})$ to be the Minkowski weight corresponding to  $$P_{k,n+1}-v\subset \{\sum x_i=0\}$$
where $v\in \{\sum x_i=k\}\cap \mathbb{Z}^{n+1}$ is any vector.
\end{defn}
We will need the following even more explicit description of the Minkowski weight $\gamma_k$.
\begin{thm}\label{thm:gammakmink} For $\mathcal{F}\in \mathcal{L}^{n+1}_{(n-1)}$ we have
\[\gamma_k(\sigma_{\mathcal{F}})=\begin{cases}0 &\exists F\in \mathcal{F}\text{ with }|F|=k\\1&\text{ otherwise.}\end{cases}\]
\end{thm}
\begin{proof}
Because the edges of $P_{k,n+1}$ are translates of $e_i-e_j$ for various $i,j$, they all have normalized volume $1$, so all weights of $\gamma_k$ are $0$ or $1$.

The maximal cones $\sigma_{\mathcal{F}}\in \Delta_{n+1,(n)}$ taking maximal value on the vertex $e_A$ are precisely those with $A\in \mathcal{F}$. Hence the cones $\sigma_{\mathcal{F}}\in \Delta_{n+1,(n-1)}$ taking maximal value on the edge $e_{B\cup i}e_{B\cup j}$ are those cones which are subcones of both a cone in $\Delta_{n+1,(n)}$ taking a maximal values at $e_{B\cup i}$, and a cone in $\Delta_{n+1,(n)}$ taking a maximal value at $e_{B\cup j}$. This happens precisely when $\mathcal{F}\in \mathcal{L}^{n+1}_{(n-1)}$ contains both $B$ and $B\cup \{i,j\}$, but no set of size $k$. Ranging over subsets $B$ and elements $i,j\not \in B$, the $\mathcal{F}$ range over all chains in $\mathcal{L}^{n+1}_{(n-1)}$ not containing a set of size $k$, and we obtain the desired description of the Minkowski weight.
\end{proof}

\begin{rmk}
The nef divisors $\alpha_M$ and $\beta_M$ defined in \cite[Definition 5.7]{AHK18}, which are used to deduce the log-concavity of the reduced characterisitic polynomial of a matroid, can be described as image in $A^\bullet(M)$  of $\gamma_n$ and $\gamma_1$, respectively, under the composite map $MW^\bullet(\Delta_{n+1})\cong A^\bullet(U_{n+1})\to A^\bullet(M)$.
\end{rmk}

Using this, we can describe how $\gamma_i$ classes restrict under the map from \Cref{starprop}.
\begin{lem}
\label{restrict}
For $\mathcal{F}\in \mathcal{L}^{n+1}_{(c)}$, the restriction $\operatorname{res}_{\mathcal{F}}(\gamma_k)$ is given in the ring  $$MW^\bullet(\Delta_{F_1})\otimes MW^\bullet(\Delta_{F_2\setminus F_1})\otimes \cdots \otimes MW^\bullet(\Delta_{\{0,\ldots,n\}\setminus F_c}),$$ by the formula,
$$\operatorname{res}_{\mathcal{F}}(\gamma_k)=\begin{cases}0&k=|F_j|\text{ for some $0\le j\le c$}\\
1^{\otimes j}\otimes \gamma_{k-|F_j|}\otimes 1^{c-j}& |F_j|<k<|F_{j+1}|\text{ for some $0\le j \le c$.}\end{cases}.$$
\end{lem}

\begin{proof}
Let $\mathcal{G}_i\in \mathcal{L}^{F_{i+1}\setminus F_i}_{(g_i)}$ with $\sum g_i=n-c-1$. We have
$$\operatorname{res}_{\mathcal{F}}(\gamma_k)(\bigotimes_{i=0}^c \sigma_{\mathcal{G}_i})=\gamma_k(\sigma_{\mathcal{H}})=\begin{cases}0&\exists H\in \mathcal{H}\text{ with }|H|=k\\1&\text{otherwise}\end{cases}$$
where $\mathcal{H}=\mathcal{F}\cup \bigcup_{i=0}^c \{G\cup F_i: G\in \mathcal{G}_i\}\in \mathcal{L}^{n+1}_{(n-1)}$. Immediately from this we see that if $k=|F_j|$ for some $j$ then $\gamma_k(\sigma_{\mathcal{H}})=0$. Otherwise, suppose we have $|F_j|<k<|F_{j+1}|$ for some (unique) $0\le j \le c$. Because the elements in $\{G\cup F_i:G\in \mathcal{G}_i\}$ have sizes between $|F_i|$ and $|F_{i+1}|$, the condition that there exists a set in $\mathcal{H}$ of size $k$ is exactly the same as the condition that there exists a set in $\mathcal{G}_j$ of size $k-|F_j|$, and the result follows.
\end{proof}

\end{comment}

\section{Products as Minkowski weights and mixed Eulerian numbers}
\label{AlphaSec2}
In this section, we describe arbitrary products of $\gamma_k$ classes as Minkowski weights on $\Delta_{n+1}$. The values taken by these Minkowski weights are products of mixed Eulerian numbers, a direct consequence of the fact that the divisor $\gamma_k$ corresponds to the $k$th hypersimplex.
\subsection{Mixed Eulerian numbers, $\gamma_i$ products, and symmetrized Minkowski weights}
First we recall mixed Eulerian numbers \cite[Section~16]{Postnikov} and use them to define a certain ``cyclic shift'' generating function.
\begin{defn}
The \emph{mixed Eulerian number} $A_{a_1,\ldots,a_k}\in \mathbb{Z}_{>0}$ is defined to be
$$n!\operatorname{mvol}(P_{1,n+1}-v_1,\ldots,P_{1,n+1}-v_1,\ldots,P_{k,n+1}-v_k,\ldots,P_{k,n+1}-v_k)$$
where $v_k\in \{(x_0,\ldots,x_n)\in \mathbb{Z}^{n+1}:x_0+\cdots+x_n=k\}$, $n=\sum a_i$, mixed volume is normalized with respect to $\{\sum x_i=0\}\cap \mathbb{Z}^{n+1}$, and $P_{i,n+1}-v_i$ appears $a_i$ times.
\end{defn}


\begin{lem}
\label{toppoweralphak}
If $r_1+\dots+r_n=n$, then $$\deg(\prod \gamma_i^{r_i})=A_{r_1,\ldots,r_n}.$$
\end{lem}
\begin{proof}
The result follows from the definition of the mixed Eulerian numbers as a mixed volume of hypersimplices and the definition of $\gamma_i$.
\end{proof}
For our later considerations, it will be crucial to understand products $\prod \gamma_i^{r_i}\in MW^r(\Delta_{n+1})$ where $r=\sum r_i$ may be smaller than $n$. To express our results, we now construct a basis for the $S_{n+1}$-invariant part of $MW^r(\Delta_{n+1})$ indexed by sequences $S:=1\le s_1 < s_2 <\ldots < s_c \le n$ as follows.
\begin{defn}
Let $S=\{s_1<\cdots<s_c\}\subset \{1,\ldots,n\}$. Define $\mathcal{L}^{n+1}_{(c),S}=\{\mathcal{F}\in \mathcal{L}^{n+1}_{(c)}: |F_i|=s_i\text{ for all $i$}\}$, and define the function $\delta_S=\delta_{s_1,\ldots,s_c}:\Delta^{n+1}_{(c)}\to \mathbb{Z}$ by $$
\delta_S(\sigma_{\mathcal{F}})=\begin{cases}1 & \mathcal{F}\in \mathcal{L}^{n+1}_{(c),S}\\0& \text{otherwise.}\end{cases}
$$
We call $\delta_S$ a \emph{symmetrized Minkowski weight}.
\end{defn}

For example, by \Cref{thm:gammakmink} we have $\gamma_j=\delta_{1,\ldots,k-1,k+1,\ldots,n}$.
\begin{prop}
The $\delta_S$ are Minkowski weights, and for $|S|=c$ are a basis for the abelian group $MW^{n-c}(\Delta_{n+1})^{S_{n+1}}$.
\end{prop}
\begin{proof}
Provided the $\delta_S$ are Minkowski weights, they clearly generate the group $MW^{n-c}(\Delta_{n+1})^{S_{n+1}}$ with no relations between them. To verify the balancing condition, it suffices to show that for every $\mathcal{F}\in \mathcal{L}^{n+1}_{(c)}$, and for each $t$ and $0\le \ell\le c$ we have
$$\sum_{|A|=t,F_\ell \subset A \subset F_{\ell+1}}e_{A} \in N_{\sigma_{\mathcal{F}}}.$$
 By symmetry this sum is clearly a linear combination of $e_{F_\ell}$ and $e_{F_{\ell+1}}$, which both lie in $N_{\sigma_{\mathcal{F}}}$, so the result follows.
\end{proof}
\begin{defn}
For a sequence of natural numbers $r_1,\ldots,r_n$ with $\sum_{i=1}^nr_i=r$, say that another sequence $1\le s_1\le \cdots \le s_{n-r}\le n$ is \emph{adapted} to the sequence $r_1,\ldots,r_n$ if for all $0\le j\le n-r$ we have
$$\sum_{s_j<k<s_{j+1}}r_k=\#\{k:s_j<k<s_{j+1}\}=s_{j+1}-s_j-1,$$
where we set by convention $s_0=0$ and $s_{n-r}=n+1$.
\end{defn}
\begin{rmk}
\label{rmk:adapted}
With the setup as in the above definition, for adaptedness to hold we must have $r_{s_j}=0$ for $1\le j \le n-r$. Indeed, summing all of the adaptedness equalities, we have
$$\sum_{k \not \in \{s_1,\ldots,s_{n-r}\}} r_k=\sum_{j=1}^{n-r}(s_{j+1}-s_j-1)=(n+1)-(n-r+1)=r=\sum_{k=1}^n r_k.$$
\end{rmk}

\begin{prop}
\label{prop:gammaproduct}
If $r_1+\ldots+r_n=r\le n$ then
$$\prod_{i=1}^n\gamma_i^{r_i}=\sum(\prod_{j=0}^{n-r}A_{r_{s_j+1},\ldots,r_{s_{j+1}-1}})\delta_{s_1,\ldots,s_{n-r}}$$
where the sum is over all sequences $1\le s_1<\cdots <s_{n-r}\le n$ adapted to $r_1,\ldots,r_n$ where we set by convention $s_0=0$ and $s_{n-r+1}=n+1$, and $A_{\emptyset}=1$. In particular, the coefficient of $\delta_{s_1,\ldots,s_{n-r}}$ is zero if $r_{s_j}\ne 0$ for any $1\le i \le n-r$.
\end{prop}
\begin{proof}
Because the left hand side is $S_{n+1}$-invariant, we can write it as a $\mathbb{Z}$-linear combination
$$\sum_{1\le s_!<\cdots <s_{n-r}\le n}c_{s_1,\ldots,s_{n-r}}\delta_{s_1,\ldots,s_{n-r}}.$$
By \Cref{cor:starcor} and \Cref{restrict}, if $\mathcal{G}\in \mathcal{L}^{n-1}_{(n-r)}$ has $|G_i|=s_i$ for all $i$, then
\begin{align*}c_{s_1,\ldots,s_{n-r}}&=(\prod_{i=1}^{n} \gamma_i^{r_i})(\sigma_{\mathcal{G}})\\&=\deg \operatorname{res}_{\mathcal{G}}(\prod_{i=1}^n \gamma_i^{r_i})=\deg \prod_{i=1}^n \operatorname{res}_{\mathcal{G}}(\gamma_i)^{r_i}\\&=\deg ((\bigotimes_{j=0}^{n-r-1} \gamma_0^{r_{s_j+1}}\cdots \gamma_{s_j-s_{j-1}}^{r_{s_{j+1}-1}}0^{r_{s_{j+1}}})\otimes (\gamma_0^{r_{s_{n-r}+1}}\cdots \gamma_{n+1-s_{n-r}}^{r_n}))\end{align*}
where by convention $0^0=1$ and $0^r=0$ for $r\ge 1$, and the degree is computed in $\bigotimes_{j=0}^{n-r} MW^\bullet(\Delta_{G_{j+1}\setminus G_j})$.
Because $$\sum_{j=0}^{n-r} \dim \Delta_{G_{j+1}\setminus G_j}=\sum_{j=0}^{n-r} (s_{j+1}-s_j-1)=(n+1)-(n-r+1)=r=\sum_{i=1}^n r_i,$$ if $s_1,\ldots,s_{n-r}$ is not adapted to $r_1,\ldots,r_n$ then for some $j$ we have that either $r_{s_j}=0$ for some $j$, or $\sum_{s_j<k<s_{j+1}} r_k>\dim \Delta_{G_{j+1}\setminus G_j}$ which again makes the product zero. Otherwise if  $s_1,\ldots,s_{n-r}$ is adapted to $r_1,\ldots,r_n$ then $r_{s_j}=0$ for $1\le j \le n-r$ by \Cref{rmk:adapted}, and the sum of the remaining exponents in the $j$'th factor is equal to $\dim \Delta_{G_{j+1}\setminus G_j}$ for all $j$. The result then follows from \Cref{toppoweralphak} applied to each factor.
\end{proof}
\begin{exmp}
Suppose that $n=4$ and we multiply $\gamma_2^2\gamma_4 $ in $A^{\bullet}(X(\Delta_5))$. Then $(r_1,r_2,r_3,r_4)=(0,2,0,1)$, $r=3$, and we can determine the coefficients of
\begin{align*}
     \gamma_2^2\gamma_4=c_1\delta_{1}+c_2\delta_{2}+c_3 \delta_{3}+c_4\delta_{4}.
\end{align*}
The sequences $(s_1)$  that are adapted to $(0,2,0,1)$ are $(1)$ and $(3)$. It follows that $c_2 = c_4 =0$, while $c_1 = A_{\emptyset}A_{r_2,r_3,r_4}=A_{2,0,1} = 3$ and $c_3 = A_{r_1,r_2}A_{r_4}=A_{0,2}A_1=1$.
%
\end{exmp}

\begin{cor}
\label{cor:deltaSasprodofalphai}
For $S=\{s_1<\cdots < s_c\}\subset \{1,\ldots,n\}$, $$\delta_S=\frac{1}{(s_1-1)!(s_2-s_1-1)!(s_3-s_2-1)!\ldots (n-s_c)!}\prod_{i\not \in S}\gamma_i\in MW^{n-r}(\Delta_{n+1})_{\mathbb{R}}.$$
In particular, (setting $\gamma_0=\gamma_{n+1}=0$ by convention)
\begin{itemize}
    \item We have $\deg(\gamma_1\cdots \gamma_n)=n!$, and for $1\le 1 \le n$ we have $\gamma_i^2=\frac{1}{2}(\gamma_{i}\gamma_{i+1}+\gamma_i\gamma_{i-1})$.
\item If $i\in S$, then with $\{a,\ldots,b\}$ the largest interval with $S\cap \{a,\ldots,b\}=\{i\}$, we have $$\gamma_i\delta_S=\frac{(b-a+1)!}{(i-a)!(b-i)!}\delta_{S\setminus i}.$$
\item If $i\not\in S$, then with $\{a,\ldots,b\}$ the largest interval containing $i$ which is disjoint from $S$ we have $$\gamma_i\delta_S=(\frac{b+1-i}{b-a+2}\gamma_{a-1}+\frac{i-a+1}{b-a+2}\gamma_{b+1})\delta_{S},$$ reducing to the previous case.
\end{itemize}
\end{cor}
\begin{proof}
The only sequence adapted to  $r_1,\ldots,r_n$ where $r_i=1$ when $i\in \{s_1,\ldots,s_c\}$ and $i=0$ otherwise is the sequence $s_1,\ldots,s_c$ itself. Hence we have $\prod_{i\not\in \{s_1,\ldots,s_c\}}\gamma_i=(\prod_{i=0}^c A_{1^{s_{i+1}-s_i-1}})\delta_S$ where we set $s_0=0$ and $s_{c+1}=n+1$ by convention, and $A_{1^{r}}=r!$ \cite[Theorem 16.3(7)]{Postnikov}. Using this, we deduce the formulas as follows.
\begin{itemize}
    \item $\deg(\gamma_1\cdots \gamma_n)=n!\deg(\delta_{\emptyset})=n!$. The only sequences adapted to the sequence $r_1=\cdots=r_{i-1}=0$, $r_i=2$, $r_{i+1}=\cdots = r_n=0$ are the sequences $S_1=\{1<\cdots<i-2<i+1<\cdots<n\}$ and $S_2=\{1<\cdots<i-1<i+2<\cdots<n\}$. Because $A_{0,2}=A_{2,0}=1$, $\gamma_i\gamma_{i-1}=\frac{1}{2}\delta_{S_1}$, and $\gamma_i\gamma_{i+1}=\frac{1}{2}\delta_{S_2}$, we obtain $\gamma_i^2=\frac{1}{2}(\gamma_{i}\gamma_{i+1}+\gamma_i\gamma_{i-1})$.
    \item This immediately follows by writing $\delta_S=c_S\prod_{j\not \in S}\gamma_j$ and $\delta_{S\setminus i}=c_{S\setminus i}\prod_{i\ne j\not \in S}\gamma_j$ and comparing the constants.
    \item By the first point, for $a \le i \le b$ and the expression $\delta_S=c\prod_{j\not\in S}\gamma_j$ we have the linear relations $\delta_S\gamma_i=\frac{1}{2}(\delta_S\gamma_{i-1}+\delta_S\gamma_{i+1})$. This is solved uniquely by the arithmetic progression $\delta_S\gamma_i=\frac{b+1-i}{b-a+2}\delta_S\gamma_{a-1}+\frac{i-a+1}{b-a+2}\delta_S\gamma_{b+1}$ interpolating from $\delta_S\gamma_{a-1}$ to $\delta_S\gamma_{b+1}$.
\end{itemize}
\end{proof}

\subsection{Contiguous $\gamma_i$ products and one-window symmetrized Minkowski weights}
Of particular interest for us are contiguous products of $\gamma_i$, as they can be expressed as linear combinations of what we call ``one-window symmetrized Minkowski weights'', which we will later show compute the coefficients of $T_M(1,y)$ when intersecting with $[\Delta_M]$.

\begin{defn}
For $0\le k \le n-r$, define $$\Phi_{r,k}=\delta_{1,\ldots,k,k+r+1,\ldots,n}\in MW^r(\Delta_{n+1}).$$ We call $\Phi_{r,k}$ a \emph{one-window} symmetrized Minkowski weight.
\end{defn}
For example, $\gamma_k=\Phi_{1,k-1}$ by \Cref{thm:gammakmink}. Our next result follows from \Cref{cor:deltaSasprodofalphai}.
\begin{cor}
\label{cor:OneWindowasprodofalphai}
For $0\le k \le n-r$, if $k\le i\le k+r+1$ then we have $$\gamma_i \Phi_{r,k}=
(k+r+1-i)\Phi_{r+1,k-1}+(i-k)\Phi_{r+1,k}$$
where we set the first term to be zero when $k=0$ and the second term to be zero when $k=n-r$.
\end{cor}

To package the computations for contiguous products of $\gamma_i$ classes, we introduce the following polynomial encoding the mixed Eulerian numbers for the shifts of a sequence.
\begin{defn}
For $a_1,\ldots,a_k \ge 1$ with $\sum a_i=r$, we define the \emph{mixed Eulerian polynomial} as
$$A_{a_1,\ldots,a_k}(y)=\sum_{i=0}^{r-k}A_{0^i,a_1,\ldots,a_k,0^{r-k-i}}y^i.$$
\end{defn}

When $k=1$ the following is a known formula for the Eulerian polynomial, as mentioned in the introduction.

\begin{thm}\label{thm:mixedEuleriangenerating} For $a_1,\ldots,a_k \ge 1$ with $\sum a_i=r$ we have
$$A_{a_1,\ldots,a_k}(y)=(1-y)^{r+1}\sum_{i=0}^{\infty}(i+1)^{a_1}(i+2)^{a_2}\ldots (i+k)^{a_k} y^i.$$
\end{thm}
\begin{rmk}
Nadeau and Tewari independently observed this identity, which later appeared in \cite[Proposition 4.5]{NadeauTewari}. For $k=1$ this recovers the known identity for Eulerian polynomials $\frac{A_r(y)}{(1-y)^{r+1}}=\sum_{i=0}^{\infty} (i+1)^r y^i$.
\end{rmk}
\begin{proof}
The result is true for $k=1$ and $a_1=1$, so suppose now that we know the result for $(a_1,\ldots,a_k)$. It then suffices to show the result for $(a_1,\ldots,a_{k-1},a_k+1)$ and $(a_1,\ldots,a_k,1)$.

The latter case is proved identically to the former case, so we omit the proof. In the former case, writing $a_1+\dots+a_k=r$, this is saying
$$\frac{A_{a_1,\ldots,a_k+1}(y)}{(1-y)^{r+2}}=y^{-k+1}\frac{d}{dy}\left(y^{k}\frac{A_{a_1,\ldots,a_k}(y)}{(1-y)^{r+1}}\right).$$
Rearranging, this is
$$A_{a_1,\ldots,a_{k}+1}(y)=(k(1-y)+(r+1)y)A_{a_1,\ldots,a_k}(y)+y(1-y)\frac{d}{dy}A_{a_1,\ldots,a_k}(y).$$
Taking the $x^i$ coefficient of both sides, we need to show that
$$A_{0^i,a_1,\ldots,a_k+1,0^{r+1-k-i}}=(k+i)A_{0^i,a_1,\ldots,a_k,0^{r-k-i}}+(r+2-(k+i))A_{0^{i-1},a_1,\ldots,a_k,0^{r+1-k-i}}.$$
Now, we note that in $A^\bullet(\Delta_{r+2})$ we have by \Cref{alphapower} that
$$\gamma_{i+1}^{a_1}\ldots \gamma_{i+k}^{a_k}=A_{0^i,a_1,\ldots,a_k,0^{r-k-i}}\Phi_{r,0}+A_{0^{i-1},a_1,\ldots,a_k,0^{r+1-k-i}}\Phi_{r,1}.$$
Multiplying both sides by $\gamma_{i+k}$, we conclude by taking the degrees of both sides, using \Cref{toppoweralphak} for the left hand side and \Cref{cor:OneWindowasprodofalphai} for the right hand side.
\end{proof}

Finally, the following is a corollary of \Cref{prop:gammaproduct}.
\begin{cor}
\label{alphapower}
For $r_{1},\ldots,r_k\ge 1$ with $\sum_{i=1}^k r_i=r$ we have in the ring $MW^\bullet(\Delta_{n+1})[y]$ that
\begin{align*}\sum_{i=0}^{n-k}\gamma_{1+i}^{r_1}\ldots\gamma_{k+i}^{r_k}y^i&=\left(\sum_{i=0}^{n-r}\Phi_{r,i}y^i\right)A_{r_{1},\ldots,r_{k}}(y).\end{align*}
\end{cor}
\begin{proof}
By comparing coefficients, we see that the statement is equivalent to showing
for $a\le b$ and $1\le r_a',\ldots,r_b'$ with $r=\sum_{i=a}^b r_i'$, we have
$$\prod_{i=a}^b \gamma_i^{r_i'}=\sum_{\{a,\ldots,b\}\subset \{i+1,\ldots,i+r\}}A_{0^{a-(i+1)},r_a',\ldots,r_b',0^{(i+r)-b}}\Phi_{r,i}.$$

Extend $r_a',\ldots,r_b'$ to a sequence $r_1',\ldots,r_n'$ by setting $r_i'=0$ if $i\not\in \{a,\ldots,b\}$. We may now apply \Cref{prop:gammaproduct}, obtaining an expression
$$\prod_{i=a}^b \gamma_i^{r_i'}=\sum c_{s_1,\ldots,s_{n-r}}\delta_{s_1,\ldots,s_{n-r}}$$
where the sum is over sequences $1 \le s_1 < \cdots < s_{n-r}\le n$ adapted to $r_1',\ldots,r_n'$. By \Cref{rmk:adapted}, we know that $s_1,\ldots,s_{n-r}\not\in \{a,\ldots,b\}$. If for some $0\le j \le n-r$ we have either $s_j,s_{j+1}\in \{0,\ldots,a-1\}$ or $s_j,s_{j+1}\in \{b+1,\ldots,n+1\}$, then $$0=\sum_{s_j<k<s_{j+1}}r_j'=s_{j+1}-s_j-1$$
and hence $s_{j+1}=s_j+1$. Therefore all but one of the differences $s_{j+1}-s_j$ is required to equal $1$, so the only possibility for an adapted sequence is if $\{s_1,\ldots,s_{n-r}\}=\{1,\ldots,i\}\sqcup \{i+r+1,\ldots, n\}$ for some $0\le i \le n-r$, and so $\delta_S=\Phi_{r,i}$. The only $1\le j \le n-r$ where adaptedness could possibly fail is when $j=i$ (so $s_j=i$ and $s_{j+1}=i+r+1$), and we see that
$$\sum_{i<k<i+r+1}r_k'=r=\sum_{i=0}^b r_i',$$ which happens precisely when $\{a,\ldots,b\}\subset \{i+1,\ldots,i+r\}$. Finally, for such an $S$, $$c_{s_1,\ldots,s_{n-r}}=\prod_{j=0}^{n-r} A_{r_{s_j+1}',\ldots,r_{s_{j+1}'-1}}=A_{\emptyset}^{n-r}A_{r_{i+1}',\ldots,r_{i+r}'}=A_{0^{a-i-1},r_a',\ldots,r_b',0^{(i+r)-b}}.$$
\end{proof}

\section{The sliding sets problem}\label{sec:sliding sets}
Let $S=\{s_1<\cdots < s_{n-r}\}\subset \{1,\ldots,n\}$ and set by convention $s_0=0$ and $s_{n-r+1}=n+1$. In this section, we introduce a combinatorial framework that we will use to compute degrees $\deg(\Psi \cdot \delta_S)$ for a Minkowski weight $\Psi\in MW^{n-r}(\Delta_{n+1})$.

\subsection{The setup of sliding sets}
Choose a generic vector $v=(v_0,\ldots,v_n)\in N_{\mathbb{R}}$ (in particular, with $v_i\ne v_j$ for all $i,j$). The idea is that we will partition the set $\{v_0,\ldots,v_n\}$ according to the partition $\bigsqcup_{i=0}^r F_{i+1}\setminus F_{i}=\{0,\ldots,n\}$, and then count how many ways we can ``slide'' these sets to the right so that the amount a set moves is decreasing in $i$, and the resulting points group up in a way determined by $S$.

\begin{defn}
Define $\operatorname{Mult}_m(\mathbb{R})$ to be the set of all total multiplicity $m$ multisets in $\mathbb{R}$.
For a finite multiset $T$ of $\mathbb{R}$, write $\operatorname{supp}(T)=\{x_0>\cdots > x_k\}$ for the underlying set. If $x_i$ has multiplicity $a_i$, we write $\operatorname{mult}_{\to}(T)=(a_k,\ldots,a_0)$ for the multiplicities read  forwards on the real line.
\end{defn}
\begin{defn}
For $\mathcal{F}\in \mathcal{L}^{n+1}_{(r)}$ and $i=0,\ldots,r$, let $H_{i}=H_i^v(\mathcal{F})=\{v_j\}_{j\in F_{i+1}\setminus F_{i}}\in \operatorname{Mult}_{|F_{i+1}\setminus F_{i}|}(\mathbb{R})$ (where by convention $F_0=\emptyset$ and $F_{r+1}=\{0,\ldots,n\}$).
\end{defn}
\begin{defn}
For $\mathcal{F}\in \mathcal{L}^{n+1}_{(r)}$, let $f^v_S(\mathcal{F})$ be the number of solutions to the ``sliding sets problem'' for $\mathcal{F}$ and $S$. That is, $f^v_S(\mathcal{F})$ is the number of choices of real numbers $t_0>\cdots>t_{r}=0$ such that $$\operatorname{mult}_{\to}(\bigcup (H_i+t_i))=(s_{n-r+1}-s_{n-r},\ldots,s_1-s_0).$$
\end{defn}
\begin{exmp}
Let $v=(1,0,3,5,10,11,12.5)$, $S=2<3<4<6$ and $\mathcal{F}\in \mathcal{L}^7_{(2)}$ with $F_1=\{1\}$, $F_2=\{1,0,2,4\}$. Then $H_0=\{0\}$, $H_1=\{1,3,10\}$, and $H_2=\{5,11,12.5\}$. We want to slide the sets $H_0,H_1,H_2$ to the right so that $H_0$ is displaced right the greatest, $H_1$ is displaced right less so and $H_2$ is stationary, and we want the final multiplicities after sliding to be $(s_5-s_4,s_4-s_3,s_3-s_2,s_2-s_1,s_1-s_0)=(1,2,1,1,2)$.
We depict the $H$'s situated initially as follows, with $H_0$ black (on the line), $H_1$ red (above), and $H_2$ blue (below).
\begin{center}
\begin{tikzpicture}
\draw (-0.5,0) -- (10.5,0);
\filldraw (0,0) circle (2pt);
\filldraw[red] (0.5,0.2) circle (2pt);
\filldraw[red] (1.5,0.2) circle (2pt);
\filldraw[red] (5,0.2) circle (2pt);
\filldraw[blue] (2.5,-0.2) circle (2pt);
\filldraw[blue] (5.5,-0.2) circle (2pt);
\filldraw[blue] (6.25,-0.2) circle (2pt);
\end{tikzpicture}
\end{center}
None of the pairs of differences of elements in $H_1$ equal a pair of difference in $H_2$, so after sliding we must have one of the points of multiplicity $2$ being a red/blue point, and the other point containing black.

The valid slides are
$(t_0,t_1,t_2)\in\{(5,2.5,0),(12.5,2,0),(20,10,0)\}$, yielding the following final configurations:
\begin{center}
\begin{tikzpicture}
\draw (-0.5,0) -- (10.5,0);
\draw (0,0) circle (0pt);
\filldraw (2.5,0) circle (2pt);
\filldraw[red] (1.75,0.2) circle (2pt);
\filldraw[red] (2.75,0.2) circle (2pt);
\filldraw[red] (6.25,0.2) circle (2pt);
\filldraw[blue] (2.5,-0.2) circle (2pt);
\filldraw[blue] (5.5,-0.2) circle (2pt);
\filldraw[blue] (6.25,-0.2) circle (2pt);

\draw (-0.5,-1) -- (10.5,-1);
\filldraw (6.25,-1) circle (2pt);
\filldraw[red] (1.5,-0.8) circle (2pt);
\filldraw[red] (2.5,-0.8) circle (2pt);
\filldraw[red] (6,-0.8) circle (2pt);
\filldraw[blue] (2.5,-1.2) circle (2pt);
\filldraw[blue] (5.5,-1.2) circle (2pt);
\filldraw[blue] (6.25,-1.2) circle (2pt);

\draw (-0.5,-2) -- (10.5,-2);
\filldraw (10,-2.0) circle (2pt);
\filldraw[red] (5.5,-1.8) circle (2pt);
\filldraw[red] (6.5,-1.8) circle (2pt);
\filldraw[red] (10,-1.8) circle (2pt);
\filldraw[blue] (2.5,-2.2) circle (2pt);
\filldraw[blue] (5.5,-2.2) circle (2pt);
\filldraw[blue] (6.25,-2.2) circle (2pt);
\end{tikzpicture}
\end{center}
Therefore $f^v_S(\mathcal{F})=3$.
\end{exmp}
\begin{prop}
\label{prop:finitefs}
$f^v_S(\mathcal{F})$ is always finite.
\end{prop}
\begin{proof}
Suppose $\operatorname{mult}_{\to}(\bigcup (H_i+t_i))=(s_{n-r+1}-s_{n-r},\ldots,s_1-s_0)$. Fix a spanning tree $T_i$ on $H_i$ for each $i$. Consider the graph $G=\bigcup_i (T_i+t_i)$ on $\operatorname{supp}(\bigcup_i (H_i+t_i))$. By the genericity of $v$, $|(H_i+t_i)\cap (H_j+t_j)|\le 1$ for all $i\ne j$, so the trees $T_i+t_i$ are are edge-disjoint. By genericity again, $G$ cannot have a cycle, and so is a forest. This forest has $\sum_{i=0}^r (|H_i|-1)=n-r$ edges, and vertex set of size $n-r+1$, so is a spanning tree. But then just from the combinatorial data of the spanning trees $T_i$ and which pairs $(v_k,v_\ell)\in T_i\times T_j$ have $v_k+t_i=v_\ell+t_j$, we can recover all of the real numbers $t_i$ since $t_{r}=0$ and $G$ is connected.
\end{proof}
\begin{rmk}
In \Cref{SlidingSetsGeneral}, we will show that $f^v_S(\mathcal{F})$ is a finite set counting the numbers of intersections between certain pairs of cones, which also establishes that it is finite.
\end{rmk}
\subsection{Sliding sets when $\delta_S$ is a one-window symmetrized Minkowski weight}
One particular case of the sliding sets problem will be of special interest to us, namely when $\delta_S=\Phi_{r,k}$ is a one-window symmetrized Minkowski weight.
\begin{thm}
\label{thm:generalonewindow}
Let $S=\{1,\ldots,k,k+r+1,\ldots,n\}$, and suppose that $v_0>v_1>\cdots>v_n$. Then for $\mathcal{F}\in \mathcal{L}^{n+1}_{(r)}$, we have
\begin{align*}f^v_S(\mathcal{F})=|\{(x_0,\ldots,x_r)\in \prod_{i=0}^r F_{i+1}\setminus F_i: \sum_{i=0}^r |\{y\in F_{i+1}\setminus F_i:y<x_i\}|=k&\\\text{and }x_0>\cdots > x_r\}&|.\end{align*}
\end{thm}
\begin{proof}
The sliding sets problem has $r+1$ sets $H_i=\{v_j\}_{j\in F_{i+1}\setminus F_{i}}$ for $0\le i \le r$, and we want to achieve multiplicities
$$\operatorname{mult}_{\rightarrow}(\bigcup (H_i+t_i))=(\underbrace{1,\ldots,1}_{n-r-k},r+1,\underbrace{1,\ldots,1}_{k}).$$
Because for fixed $i$ no two elements of $H_i+t_i$ can be equal, there must be one point in each $H_i+t_i$ which is part of this group of multiplicity $r+1$. The genericity of $v$ ensures that no further overlaps occur, and if $x_i\in F_{i+1}\setminus F_{i}$ are the chosen points so that $v_{x_0}+t_0=v_{x_1}+t_1=\cdots=v_{x_{r}}+t_{r}$, then the number of points in $\bigcup (H_i+t_i)$ beyond the point of multiplicity $r+1$ is exactly $\sum_{i=0}^{r}|\{y \in F_{i+1}\setminus F_{i}: v_y > v_{x_i}\}|=\sum_{i=0}^{r}|\{y \in F_{i+1}\setminus F_{i}: y < x_i\}|$. Because $v_{x_0}+t_0=\cdots=v_{x_{r}}+t_{r}$, the condition $x_0>\cdots>x_{r}$, or equivalently, $v_{x_0}<\cdots <v_{x_{r}}$ is precisely the condition that $t_0>t_1>\cdots >t_{r}$.
\end{proof}

\section{Intersections with Symmetrized Minkowski weights via sliding sets}
\label{SlidingSetsGeneral}
In this section we show how the sliding sets problem can be used to compute intersections with symmetrized Minkowski weights. 
As in the previous section, let $S=\{s_1<\cdots<s_{n-r}\}$ (we set $s_0=0$ and $s_{n-r+1}=n+1$ by convention in what follows) and $\Psi\in MW^{n-r}(\Delta_{n+1})$ be an arbitrary Minkowski weight.

\begin{thm}
For a generic vector $v\in N_{\mathbb{R}}$, we have $$\deg(\delta_S\cdot \Psi)=\sum_{\mathcal{F}\in \mathcal{L}^{n+1}_{(r)}}f_S^v(\mathcal{F})\Psi(\sigma_{\mathcal{F}}).$$
\end{thm}

\begin{proof}
By the fan displacement rule and \Cref{lem:mult1}, we have
\begin{align*}
    \deg(\delta_S \cdot \Psi)&=\sum_{(\sigma_1,\sigma_2)\in \Delta_{n+1,(n-r)}\times \Delta_{n+1,(r)}}m^{\{0\},v}_{\sigma_1,\sigma_2}\delta_{S}(\sigma_1)\Psi(\sigma_2)\\
    &=\sum_{\mathcal{F}\in \mathcal{L}^{n+1}_{(r)}}\left(\sum_{\mathcal{G}\in \mathcal{L}^{n+1}_{(n-r),S}} m^{\{0\},v}_{\sigma_{\mathcal{G}},\sigma_{\mathcal{F}}}\right)\Psi(\sigma_{\mathcal{F}})\\
    &=\sum_{\mathcal{F}\in \mathcal{L}^{n+1}_{(r)}}|\{\mathcal{G}\in \mathcal{L}^{n+1}_{(n-r),S}: (\sigma_{\mathcal{F}}+v) \cap \sigma_{\mathcal{G}} \ne \emptyset\}| \Psi(\sigma_{\mathcal{F}})
\end{align*}
Hence it suffices to show that
$$f^v_S(\mathcal{F})=|\{\mathcal{G}\in \mathcal{L}^{n+1}_{(n-r),S}: (\sigma_{\mathcal{F}}+v) \cap \sigma_{\mathcal{G}} \ne \emptyset\}|.$$
By the genericity of $v$, we may assume that all intersections happen in the relative interiors $\sigma_{\mathcal{G}}^{\circ}=\{\sum_{i=0}^{n-r}x_ie_{G_{i+1}\setminus G_i}:x_0\ge \cdots \ge x_{n-r}=0\}\subset \sigma_{\mathcal{G}}$. Let $\sigma'_{\mathcal{F}}=\{\sum_{i=0}^{r}t_ie_{F_{i+1}\setminus F_i}:t_0\ge \cdots \ge t_r=0\}\subset \mathbb{R}^{n+1}$ and $\sigma'_{\mathcal{G}}=\{\sum_{i=0}^{n-r}x_ie_{G_{i+1}\setminus G_i}:x_0\ge \cdots \ge x_{n-r}=0\}\subset \mathbb{R}^{n+1}$, so that under the quotient map $\mathbb{R}^{n+1}\to N_{\mathbb{R}}$ we have $\sigma'_{\mathcal{F}},\sigma'_{\mathcal{G}}$ biject to $\sigma_{\mathcal{F}}$ and $\sigma_{\mathcal{G}}$. Then elements of intersections $(\sigma_{\mathcal{F}}+v)\cap \sigma_{\mathcal{G}}^{\circ}$ are in bijection with elements of intersections $(\sigma_{\mathcal{F}}'+v)\cap ((\sigma_{\mathcal{G}}')^{\circ}+\mathbb{R}\langle(1,\ldots,1)\rangle)$. Because the sets $(\sigma_{\mathcal{G}}')^{\circ}+\mathbb{R}\langle(1,\ldots,1)\rangle$ are disjoint, we have
$$|\{\mathcal{G}\in \mathcal{L}^{n+1}_{(n-r),S}: (\sigma_{\mathcal{F}}+v) \cap \sigma_{\mathcal{G}} \ne \emptyset\}|=|\sigma_{\mathcal{F}}'\cap (\bigcup_{\mathcal{G}\in \mathcal{L}^{n+1}_{(n-r),s}}(\sigma_{\mathcal{G}}')^{\circ}+\mathbb{R}\langle(1,\ldots,1)\rangle)|.$$
Write $\chi(x_0,\ldots,x_n)\in \operatorname{Mult}_{n+1}(\mathbb{R})$ for the multiset associated to the sequence $x_0,\ldots,x_n$, and for $A\subset \mathbb{R}^{n+1}$ denote $\chi(A):=\{\chi(x):x\in A\}$. Then
$\chi(\bigcup_{\mathcal{G}\in \mathcal{L}^{n+1}_{(n-r),s}}(\sigma_{\mathcal{G}}')^{\circ}+\mathbb{R}\langle(1,\ldots,1)\rangle )$ is the set of all $(x_0,\ldots,x_{n})$ such that $$\chi(x_0,\ldots,x_n)=(s_{n-r+1}-s_{n-r},\ldots,s_1-s_0),$$
and
$\chi(\sigma_{\mathcal{F}})=\{\bigcup(H_i+t_i):t_0\ge \cdots \ge t_r=0\}$. Thus intersection points correspond to solutions to sliding sets problems. The correspondence is injective because $\chi|_{\sigma'_{\mathcal{F}}}$ is injective, and the correspondence is surjective by applying the identity $\chi(A\cap \chi^{-1}(B))=\chi(A)\cap B$ with $A=\sigma_\mathcal{F}'$ and $B=\bigcup (\sigma_{\mathcal{G}'})^{\circ}+\mathbb{R}\langle (1,\ldots,1)\rangle$.
\end{proof}
\begin{cor}
\label{fsvcor}
For intersecting the Bergman fan Minkowski weight $[\Delta_M]\in MW^{n-r}(\Delta_{n+1})$ and a symmetrized Minkowski weight $\delta_S\in MW^r(\Delta_{n+1})$ with $|S|=n-r$, we have the degree computation
$$\deg([\Delta_M]\cdot \delta_S)=\sum_{\mathcal{F}\in \mathcal{L}^{M}_{(r)}}f_S^v(\mathcal{F}).$$
\end{cor}
\begin{proof}
This follows from the above theorem and the definition of $[\Delta_M]\in MW^{n-r}(\Delta_{n+1})$ from \Cref{defn:BergmanMinkowski}.
\end{proof}

\section{Tutte Polynomials and one-window symmetrized Minkowski weights}
\label{PhirkSection}
Let $M$ be a loopless matroid of rank $r+1$ on $\{0,\dots,n\}$. In this section, we show that the generating function of $\deg([\Delta_M]\cdot \Phi_{r,k})$ is precisely $T_M(1,y)$. To do this, we recall that in \Cref{inexTutte} we had the explicit formula
$$T_M(1,y) = \sum_B y^{|\operatorname{ex}(B)|},$$
  the sum over bases $B$ of $M$.
\begin{thm}\label{mainphithm} For $M$ a loopless matroid of rank $r+1$ on $\{0,\ldots,n\}$ we have
$$T_M(1,y)=\sum_{k=0}^{n-r} \deg([\Delta_M]\cdot \Phi_{r,k})y^k.$$
\end{thm}

\begin{proof}
By \Cref{thm:generalonewindow} and \Cref{fsvcor}, we have
$\deg([\Delta_M]\cdot \Phi_{r,k})=\sum_{\mathcal{F}\in \mathcal{L}^{M}_{(r),S}}f_k(\mathcal{F})$
where
\begin{align*}f_k(\mathcal{F})=|\{(x_0,\ldots,x_r)\in \prod_{i=0}^r F_{i+1}\setminus F_i: \sum_{i=0}^r |\{y\in F_{i+1}\setminus F_i:y<x_i\}|=k&\\\text{and }x_0>\cdots > x_r\}&|.\end{align*}
We claim that there is a natural bijection between the $(r+1)$-tuples counted by the various $f_k(\mathcal{F})$ and bases of $M$ with external activity $k$.

First, we claim that $(x_0,\ldots,x_r)$ is a basis. Indeed, $\rk_M(\{x_0\})=1$ since $M$ is loopless, and $i+1\ge \rk_M(\{x_0,\ldots,x_{i+1}\})>\rk_M(\{x_0,\ldots,x_{i}\})$ since $x_{i+1}\not \in F_{i+1} \supset \{x_0,\ldots,x_i\}$, so by induction we have $\rk_M(\{x_0,\ldots,x_i\})=i+1$ and conclude by setting $i=r$.

Now, note that $(x_0,\ldots,x_{r})$ determines $\mathcal{F}$ because $F_{i+1}=\overline{\{x_0,\ldots,x_{i}\}}$ (as $F_{i+1}$ contains $\{x_0,\ldots,x_{i}\}$ and they have the same rank $i+1$). Therefore all of the $(r+1)$-tuples for the various $\mathcal{F}$ are disjoint. Furthermore, the unordered bases $\{x_0,\ldots,x_{r}\}$ are distinct since only one ordering has $x_0>\cdots>x_{r}$. We claim that $\mathcal{B}=\{x_0,\ldots,x_{r}\}$ has external activity $k$. Indeed, for $y \in F_{i+1}\setminus F_{i}$, if $\mathcal{B}-x_j\cup y$ is a basis then $j\le i$, as otherwise the $i+2$ elements $x_0,\ldots,x_{i},y$ of this basis lie in the rank $i+1$ flat $F_{i+1}$. Therefore, it suffices to check the external activity condition for $y$ with respect to $x_0,\ldots, x_{i}$. If $y<x_{i}$, then since $x_{i}<x_{i-1}<\cdots<x_0$, there are no more basis elements to check the condition with respect to, so $y$ is externally active. Conversely, if $y>x_{i}$, then we claim that $\mathcal{B}-x_i\cup y$ is a basis so $y$ is not externally active. We will do this by showing $F_{r+1}=\overline{\mathcal{B}-x_i\cup y}$. Indeed, $F_i\supset \overline{\{x_0,\ldots,x_{i-1},y\}}\supsetneq F_{i-1}$ so $F_i=\overline{\{x_0,\ldots,x_{i-1},y\}}$, and hence for $j\ge i$ we inductively conclude $F_j=\overline{F_{j-1}\cup x_j}=\overline{\{x_0,\ldots,x_{i-1},y,x_{i+1},\ldots,x_j\}}$. The result follows.
\end{proof}
\begin{thm}(\Cref{thm:mixedintro})\label{gammathm}
For $r_1+\ldots+r_k=r$ and $r_i\ge 1$ we have
$$\sum_{i=0}^{n-k}\deg([\Delta_M]\cdot \gamma_{1+i}^{r_1}\ldots \gamma_{k+i}^{r_k})y^i=T_M(1,y)A_{r_1,\ldots,r_k}(y).$$
\end{thm}
\begin{proof}
This follows by combining the above \Cref{mainphithm} with \Cref{alphapower}.
\end{proof}

\section{Log concavity statements}
\label{LogConcavitySection}
Let $M$ be a loopless rank $r+1$ matroid on $\{0,\dots,n\}$. 
In this section we deduce a strengthening of Dawson's conjecture \cite{Dawson} on the log concavity of the coefficients of $T_M(1,y)$, or equivalently by \Cref{fact:Dawsonfact} the $h$-vector of the independence complex of the dual matroid $M^*$. We note that since we are working with the $h$-vector of the dual matroid $M^*$, and the $h$-vector of a matroid is insensitive to the presence of coloops, our assumption that $M$ has no loops is made without loss of generality in what follows.

\begin{thm}\label{thm:stronglc}(Equivalent reformulation of \Cref{preintrothm} for $M^*$ by \Cref{fact:Dawsonfact})
For three consecutive coefficients $a,b,c$ of $T_M(1,y)$ we have $$r(b^2-ac)+(b-a)(b-c) \ge 0,\text{ and }b^2\ge ac.$$
\end{thm}
\begin{proof}
Recall that by \Cref{cor:nointernal}, the nonnegative sequence $a,b,c$ has no internal zeros.

If $\rk(M)\le 1$ then $T_M(1,y)=1+\cdots+y^{n}$, so the inequalities are trivial. If $\rk(M)=2$ (i.e. $r=1$), one can check directly that $b \ge \frac{a+c}{2}$ (equivalent to the first inequality) and in particular $b^2\ge ac$ by the arithmetic-geometric mean inequality, so assume now that $r \ge 2$.

\begin{clm}\label{cor:lc}
For a sequence of three nonnegative numbers $a,b,c$ with no internal zeros, if $r(b^2-ac)+(b-a)(b-c)\ge 0$ for some $r\ge 1$ then $b^2\ge ac$.
\end{clm}
\begin{proof}
Without loss of generality, suppose $c\ge a$. If $b\ge c$ then $b^2-ac\ge ac-ac=0$. If $a\le b \le c$ then $r(b^2-ac)\ge (b-a)(c-b)\ge 0$. If $0<b<a$ then we have the contradiction $b^2-ac < b(a+c-b)-ac=(b-a)(c-b)\le r(b-a)(c-b)$. The result follows.
\end{proof}
Hence, it suffices to check the first inequality displayed in the theorem. This inequality is equivalent to $$((r-1)a+b)((r-1)c+b) \le (rb)^2.$$ 
Consider the product of $r-2$ consecutive $\gamma$ classes $\Psi=\gamma_{\ell+1}\ldots \gamma_{\ell+r-2}$. Then we can compute
\begin{align*}
\deg([\Delta_M]\cdot \gamma_\ell^2\Psi)&=[y^{\ell-1}]T_M(1,y)A_{2,1^{r-2}}(y)=(r-1)!((r-1)c+b),\\
\deg([\Delta_M]\cdot \gamma_{\ell}\gamma_{\ell+r-1}\Psi)&=[y^\ell]T_M(1,y)A_{1^r}(y)=r!b,\\
\deg([\Delta_M]\cdot \gamma_{\ell+r-1}^2\Psi)&=[y^{\ell+1}]T_M(1,y)A_{1^{r-2},2}(y)=(r-1)!((r-1)a+b),\end{align*}
using \Cref{gammathm} where $a,b,c$ are the $y^{\ell},y^{\ell-1},y^{\ell-2}$ coefficients of $T_M(1,y)$ respectively. Here we have evaluated the mixed Eulerian polynomials
\begin{align*}
   A_{2,1^{r-2}}(y)&=A_{2,1^{r-2},0} +A_{0,2,1^{r-2}} y=(r-1)!+(r-1)(r-1)!y,\\
   A_{1^r}(y)&=A_{1^r}=r!,\\
   A_{1^{r-2},2}(y)&=A_{1^{r-2},2,0}+A_{0,1^{r-2},2}y=(r-1)(r-1)!+(r-1)!y,
\end{align*}
using, e.g., \Cref{thm:mixedEuleriangenerating}. The result now follows from \Cref{thm:MinkAHK}.
\end{proof}
\section{Schur polynomials and the reliability polynomial of $M$}
\label{KlyachkoAppendix}
In this section we consider certain Schubert and Schur polynomials in the Chow ring of the permutohedron. Using the results of \cite{K85} we are able to connect the products of these Schur polynomials with the Bergman class of a matroid to the reliability polynomial of the matroid.

Let $t_i=(i,i+1) \in S_{n+1}$ be an adjacent transposition, $1 \leq i \leq n$,  and recall the \emph{length} $\ell(\sigma)$ of a permutation $\sigma \in S_{n+1}$ is the smallest number $r$ such that we can write $\sigma=t_{i_1}\ldots t_{i_r}$ for adjacent transpositions $t_{i_j}$. Let $\text{RW}(\sigma)$ be the set of such tuples $(i_1,\ldots,i_r)$ for a given $\sigma$. There is a unique longest permutation $\sigma_0$ which reverses the numbers $1,\ldots,n+1$. Recall that to each $\sigma\in S_{n+1}$ there is an associated \emph{Schubert polynomial}, determined by the relations
\begin{align*}
    f_{\sigma_0}(x_1,\ldots,x_{n+1})&=x_1^nx_2^{n-1}\ldots x_{n}\\
    \frac{f_\sigma(x_1,\ldots,x_{n+1})-t_i f_\sigma(x_1,\ldots,x_{n+1})}{x_{i}-x_{i+1}}&=\begin{cases}f_{t_i \sigma } & \ell(t_i \sigma)<\ell(\sigma)\\0 & \ell(t_i \sigma)>\ell(\sigma).\end{cases}
\end{align*}
\begin{defn}
Let $\alpha_i\in A^1(U_{n+1})$ be the element corresponding to $\gamma_{n+1-i}\in MW^1(\Delta_{n+1})$ for $i=1,\ldots,n$ under the isomorphism $A^\bullet(U_{n+1})\cong MW^\bullet(\Delta_{n+1})$.
\end{defn}
In \cite{K85}, Klyachko considered the following evaluations of the Schubert polynomials,
$$[P_{\sigma}]=f_{\sigma}(\alpha_1,\alpha_2-\alpha_1,\ldots,\alpha_{n}-\alpha_{n-1},-\alpha_n)\in A^{\ell(\sigma)}(U_{n+1}).$$
As noted in \cite{K85}, a general torus-orbit closure in the complete flag variety $GL_{n+1}/B$ is isomorphic to the toric variety $X(\Delta_{n+1})$   defined by the fan $\Delta_{n+1}$, and $[P_{\sigma}]\in A^\bullet(X(\Delta_{n+1}))=A^\bullet(U_{n+1})$ is the Chow class of the restriction of the Schubert variety $P_{\sigma}\subset GL_{n+1}/B$ to $X(\Delta_{n+1})$.

We will consider the following question.
\begin{quest}
\label{quest:Schubert}
For $M$ a loopless rank $r+1$ matroid on $\{0,\ldots,n\}$, and $\sigma\in S_{n+1}$ a permutation with $\ell(\sigma)=r$, what is $\deg([P_{\sigma}]\cdot [\Delta_M])$? Equivalently, what is the degree of the image of $[P_\sigma]$ in $A^r(M)$?
\end{quest}
\begin{rmk}
\label{rmk:WM}
If $M$ is realized by a vector configuration over $\mathbb{C}$ given by the columns of a full row rank $(r+1)\times (n+1)$ matrix, then associated to the rowspan $L\subset \mathbb{C}^{n+1}$  is a subvariety $W_L\subset X(\Delta_{n+1})$ (a \textit{wonderful compactification}; see \cite{DCP}) with $[W_L]=[\Delta_M]$ and $A^\bullet(W_L)=A^\bullet(M)$. Hence, in the realizable case the degree $\deg([P_{\sigma}]\cdot [\Delta_M])$ can be interpreted as the intersection number of $W_L\subset GL_{n+1}/B$ with the Schubert variety $P_{\sigma}$.
\end{rmk}

We first recall a formula of Klyachko for $[P_\sigma]$.
\begin{prop}[\cite{K85} Theorem 4]\label{prop:Kly1} For $\sigma\in S_{n+1}$ with $\ell(\sigma)=r$ we have
$$[P_\sigma]=\frac{1}{r!}\sum_{(i_1,\ldots,i_r)\in \text{RW}(\sigma)}\alpha_{i_1}\cdots \alpha_{i_r}\in A^r(U_{n+1}).$$
\end{prop}
This formula implies that $[P_{\sigma}]$ is always a nonnegative Minkowski weight, so in particular $\deg([P_{\sigma}]\cdot [\Delta_M])\ge 0$  whenever $\ell(\sigma)=r$, regardless of whether $M$ is realizable.
From this formula, we can also show that the symmetrized Minkowski weights $\delta_S$ are Schubert classes themselves.
\begin{cor}
\label{cor:deltaschubert}
For $S=\{s_1<\cdots<s_{n-r}\}\subset \{1,\ldots,n\}$ we have  $\delta_S=[P_{\sigma}]\in A^r(U_{n+1})$ for $\sigma=\sigma_1\sigma_2\ldots \sigma_{n-r+1}$ where $\sigma_i$ is the forward cycle permutation of $ (s_{i-1}+1,\ldots,s_{i})$, where we set $s_0=0$ and $s_{n-r+1}=n+1$.
\end{cor}
\begin{proof}
This follows immediately from the formula for $\delta_S$ from \Cref{cor:deltaSasprodofalphai} and  \Cref{prop:Kly1}. Indeed the permutation $\sigma_i$ has length $s_i - s_{i-1}-1$ and only one reduced word. Since the $\sigma_i$ commute the reduced words for $\sigma$ are obtained by shuffling, in order, the reduced words for the $\sigma_i$.
\end{proof}
A permutation $\sigma\in S_{n+1}$ is called $p$-Grassmannian $\sigma(1)<\cdots<\sigma(p)$ and $\sigma(p+1)<\cdots<\sigma(n+1)$. Klyachko also has a more specific formula when $\sigma$ is $p$-Grassmannian.

\begin{prop}[\cite{K85} Theorems 5,6]
\label{prop:Kly2}
 Fix a partition $\lambda$ of $r$ into at most $p$ parts, and let $f_\lambda$ denote the corresponding Schur polynomial. Let $\sigma_{\lambda,p}$ be the associated $p$-Grassmannian permutation. Then,
$$[P_{\sigma_{\lambda,p}}]=\prod_{(i,j)\in \lambda}\frac{\alpha_{p-i+j}}{h_{ij}}=\frac{1}{r!}\sum_{1\le k \le r}m_k(\lambda)\prod_{i=p-k+1}^{p-k+r}\alpha_i\in A^r(U_{n+1})$$
where $h_{ij}$ is the hook length of the cell $(i,j) \in \lambda$ and $m_k(\lambda)$ is defined by the equality
$$\sum_{k=1}^{\infty}m_k(\lambda)y^k=(1-y)^{r+1}\sum_{i=1}^{\infty}f_\lambda(1^i)y^i.$$
\end{prop}
Note that by \Cref{thm:mixedEuleriangenerating} and the hook-content formula, the generating function for $m_k(\lambda)$ is a finite sum, and a multiple of the generating function for certain mixed Eulerian numbers. In private communiation, Vasu Tewari informed us that $m_k(\lambda)$ also counts semistandard fillings of the partition $\lambda$ with $k$ descents.

From this formula, we note an interesting connection between Schur polynomials and the reliability polynomial $R_M(y)=(1-y)^{r+1}y^{n-r}T_M(1,y^{-1})$.
\begin{thm}
For $\lambda$ a partition of $r$ into at most $p$ parts, we have
$$\sum \deg([\Delta_M]\cdot [P_{\sigma_{\lambda,p}}])y^p=R_M(y)\sum_{i=1}^\infty f_\lambda(1^i)y^i.$$
\end{thm}
\begin{proof}
We consider the equality
$$[P_{\sigma_{\lambda,p}}]=\frac{1}{r!}\sum_{1 \le k \le r}m_k(\lambda)\prod_{i=p-k+1}^{p-k+r} \alpha_i$$ from \Cref{prop:Kly2}.
We claim that $m_k(\lambda)=0$ for $k\ge r+1$. Indeed, for any power series $h(y)$ one can show by induction if
$[y^k](1-y)^{k}h(y)=0$
for all $k \ge r+1$ then
$[y^k](1-y)^{r+1}h(y)=0$ for all $k \ge r+1$. Recall that $f_\lambda(1^i)$ is the number of semistandard fillings of $\lambda$ with $\{1,\ldots,i\}$, or equivalently any collection of $i$ numbers. Taking $h(y)=\sum f_\lambda(1^i)y^i$, we can interpret $[y^{k}](1-y)^{k}h(y)$ as an inclusion-exclusion counting the number of semistandard tableux of shape $\lambda$ with $k$ numbers such that each number appears at least once, which is obviously $0$ when $k$ exceeds the size $r$ of $\lambda$.

We also have $\deg(\frac{1}{r!}\prod_{i=p-k+1}^{p-k+r}\alpha_i\cdot [\Delta_M])=[y^{n-p+k-r}]T_M(1,y)$ by \Cref{thm:mixedintro}. The result now follows.
\end{proof}
We note the following way to package these right hand sides for different $\lambda$.
\begin{cor}
With $x_1,\ldots,x_{n+1-p}$ indeterminates, we have
\begin{align*}&\sum_{\lambda}\deg([P_{\sigma_{\lambda,p}}]\cdot [\Delta_M])f_{\lambda}(x_1,\ldots,x_{n+1-p})\\=&[y^p]R_M(y)\frac{(y\prod(1+x_i))^{p+1}-(\prod x_i)^{p+1}}{y\prod(1+x_i)-\prod x_i},\end{align*}
where the sum is over all partitions $\lambda$ of $r$ with at most $p$ parts, and $\lambda^{\dagger}\subset p\times (n+1-p)$ is the complementary partition (so that $\lambda$ and $\lambda^{\dagger}$ rotated by $\pi$ and translated exactly fill a $p\times (n+1-p)$ box).
\end{cor}
\begin{proof}
 Specializing $z_j\in \{0,1\}$ in the Cauchy identity $$\sum_\lambda f_{\lambda^{\dagger}}(z_1,\ldots,z_p)f_\lambda(x_1,\ldots,x_{n+1-p}) =\prod_{i=1}^{n+1-p}\prod_{j=1}^{p} (x_i+z_j),$$
 we have for $0\le i\le p$
$$\sum_\lambda f_{\lambda^{\dagger}}(1^i)f_\lambda(x_1,\ldots,x_{n+1-p})=(\prod_{i=1}^{n+1-p} x_{i})^{p-i}(\prod_{i=1}^{n+1-p} (1+x_i))^{i}.$$
Therefore
\begin{align*}&\sum_{\lambda}\deg([P_{\sigma_{\lambda,p}}]\cdot [\Delta_M])f_{\lambda}(x_1,\ldots,x_{n+1-p})\\=&[y^p]R_M(y)\sum_\lambda\sum_{i=1}^\infty f_{\lambda^{\dagger}}(1^i)f_\lambda(x_1,\ldots,x_{n+1-p})y^i\\
=&[y^p]R_M(y)\frac{(y\prod (1+x_i))^{p+1}-(\prod x_i)}{y\prod (1+x_i)-\prod x_i}.\end{align*}
\end{proof}
\begin{rmk}
When we are in the situation of \Cref{rmk:WM}, then taking $x_1,\ldots,x_{n+1-p}$ to be the Chern roots of the tautological quotient bundle of the Grassmannian $Gr(p,n+1)$ of $p$-dimensional subspaces of $\mathbb{C}^n$, and $\pi_{p}:GL_{n+1}/B\to Gr(p,n+1)$ the forgetful map, the left hand side is exactly $\pi_{p,*}([W_L])\in A^\bullet(Gr(p,n+1))$.
\end{rmk}

We conclude by making the following conjecture concerning these classes and log-concavity statements.
\begin{conj}
\label{conj:SchubertLogConc}
$M$ is a loopless rank $r+1$ matroid on ground set $\{0,\ldots,n\}$, if $\ell(\sigma)=r-2$ and $H,K\in MW^1(\Delta_{n+1})$ are divisors associated to generalized permutahedra, then
$$\deg([\Delta_M]\cdot [P_{\sigma}]\cdot H^2),\deg([\Delta_M]\cdot [P_{\sigma}]\cdot HK),\deg([\Delta_M]\cdot [P_{\sigma}]\cdot K^2)$$
is a log-concave sequence.
\end{conj}
We note some special cases where this is true.
\begin{thm}
The conjecture is true when
\begin{enumerate}
    \item $\sigma$ is a Grassmannian permutation.
    \item $\sigma=\sigma_1\cdots \sigma_{r-1}$ where $\sigma_i$ is the forward cycle permutation of $(s_{i-1}+1,\ldots,s_i)$ for a sequence $\{s_1<\cdots < s_{r-2}\}\subset \{1,\ldots,n\}$, where we set by convention $s_0=0$ and $s_{r-1}=n+1$.
\end{enumerate}
\end{thm}
\begin{proof}
In the first case $P_\sigma$ is a constant multiple of a product of $\gamma_i$ classes by \Cref{prop:Kly2}, and in the second case $P_{\sigma}$ is a constant multiple of a product of $\gamma_i$ classes by  \Cref{cor:deltaschubert} and \Cref{cor:deltaSasprodofalphai}, so the result follows from \Cref{thm:MinkAHK}
\end{proof}
\appendix
\section{Recursively computing the degree of $\deg_{A^\bullet(M)}(\delta_S)$}
\label{RecursionAppendix}
In this section, we show how to take a symmetrized Minkowski weight of a matroid on ground set $E=\{0,\ldots,n\}$ and compute its degree via a generalized deletion-contraction type recurrence. For positive integers $m_0,\ldots,m_{\crk(M)}$ with $m_0+\ldots+m_{\crk(M)}=|E|$, let $g_M(m_0,\ldots,m_{\crk(M)})=\deg_{A^\bullet(M)}(\delta_S)$ where $|S|=\crk(M)$ has $s_i=\sum_{j=0}^{i-1}m_j$ for $1\le i \le \crk(M)$, i.e. so that the grouping associated to the sliding sets problem has
$$\operatorname{mult}_{\to}(\bigcup_i (H_i+t_i))=(m_{\crk(M)},\ldots,m_0)$$
(note that this is in the reverse order to the order the $m_i$ appear in $g_M$).
Consider the $\rk(M)-1$ polynomial in $\crk(M)+1$ variables $$h_M=h_M(z_0,\ldots,z_{\crk(M)})=\sum_S g_M(S)z_0^{m_0-1}z_1^{m_1-1}\ldots z_{\crk(M)}^{m_{\crk(M)}-1}.$$ Denote by $\mathcal{T}$ the operation on a polynomial in  variables $z_i$ which replaces all $z_i$ with $z_{i+1}$.
\begin{thm}
For $n$ not a coloop of $M$, we have
\begin{align*}
    h_M=h_{M\setminus n}&+1_{\text{\{n\} a flat}}h_{(M/\{n\})\cup \{n\}}\\&+\sum_{\substack{n\in F\text{ proper flat}\\n\text{ not a coloop in }M|_{F}}}h_{(M|_{F}\setminus n)}\mathcal{T}^{\crk(M|_{F})}h_{(M/F)\cup \{n\}},
\end{align*}
and for $n$ a coloop of $M$ we have
$h_{M}=(z_0+\ldots+z_{\crk(M)})h_{M/\{n\}}.$ Here for $M'$ a matroid on $\{0,\ldots,n-1\}$ we are denoting by $M'\sqcup \{n\}$ for the matroid where $n$ is adjoined to $M'$ as a coloop, i.e. the matroid on $\{0,\ldots,n\}$ with flats $\mathcal{L}^{M'\sqcup \{n\}}=\mathcal{L}^M\sqcup \{F\cup \{n\}:F\in \mathcal{L}^M\}$.
\end{thm}
\begin{rmk}
The base case of this recursion is $h_{U_{1,1}}=1$, but we also note that we have $h_{U_{1,k}}(z_1,\ldots,z_k)=1$ and $[z_j^{\rk(M)-1}]h_{M}(z_1,\ldots,z_{\crk(M)+1})=[y^{j-1}]T_M(1,y)$.
\end{rmk}
\begin{proof}
We consider the enumerations provided by \Cref{fsvcor}, and take initial vector $v=(v_0,\ldots,v_n)$ with $v_n$ minimal such that $L=\min\{v_0,\ldots,v_{n-1}\}-v_n$ is significantly larger than $\max\{v_0,\ldots,v_{n-1}\}-\min\{v_0,\ldots,v_{n-1}\}$. Because $v$ only depends up to translation by a multiple of $(1,\ldots,1)$, we may assume that $\min\{v_0,\ldots,v_{n-1}\}=0$ (so that $v_n=-L$).

Recall that a hypergraph on a set $A$ is a family of subsets of $A$. The sets in the family are called hyperedges (or just edges), and a hypergraph is connected if there isn't a nontrivial partition $A=A_1\sqcup A_2$ with every hyperedge completely contained in either $A_1$ or $A_2$. By the genericity of $v$, for any $t_0>\ldots>t_r=0$ the sets $H_j+t_j$ in a sliding sets problem contributing to $h_M$ are the hyperedges of a simple hypertree (i.e. a connected hypergraph where two hyperedges intersect in at most one element, and the removal any hyperedge disconnects the hypergraph) by the same proof as for \Cref{prop:finitefs}.

For $\mathcal{F}\in\mathcal{L}^M_{(r)}$, we define the set of \emph{candidate hypertrees} $\operatorname{Cand}(\mathcal{F})$ to be all hypertrees obtained as $H_i(\mathcal{F})+t_i$ for real numbers $t_0,\ldots,t_{r-1}$ and $t_r=0$ (we do not impose an order on the $t_i$). Let $$\operatorname{Cand}(M)=\{(\mathcal{F},T):\mathcal{F}\in \mathcal{L}^m_{(r)}\text{ and }T\in \operatorname{Cand}(\mathcal{F})\}.$$

We say that $T\in \operatorname{Cand}(\mathcal{F})$ is \emph{good} if the induced $t_i$ have $t_0>\cdots>t_r=0$. Write $\operatorname{Good}(\mathcal{F})$ for the set of good hypertrees, and $\operatorname{Good}(M)$ for the subset of $\operatorname{Cand}(M)$ where we restrict $T\in \operatorname{Good}(\mathcal{F})$ for some $\mathcal{F}\in \mathcal{L}^M_{(r)}$. Our problem is now to enumerate the hypertrees in $\operatorname{Good}(M)$ and package the result into the generating function $h_M$.

Note that for $(\mathcal{F},T)\in \operatorname{Cand}(M)$, all hyperedges $H_j+t_j$ except for $H_i+t_i$ have diameter $\ll L$, and $H_i+t_i$ has diameter $0$ if $H_i=\{n\}$ and diameter $\approx L$ otherwise.
Therefore, we have the following cases.

In the first case, we consider the subset $\operatorname{Cand}_1(M)$ of those $(\mathcal{F},T)$ such that $\{n\}=F_{i+1}\setminus F_i$ for some $i$, and enumerate the subset $\operatorname{Good}_1(M)=\operatorname{Cand}_1(M)\cap \operatorname{Good}(M)$. We now construct a natural bijection
$$\operatorname{Good}_1(M)\cong \operatorname{Good}((M/n)\cup n).$$
For $(\mathcal{F},T)\in \operatorname{Cand}_1(M)$, we have $t_i \approx L$ and $t_j \ll L$ for all $i\ne j\in \{0,\ldots,r\}$. Hence for $(\mathcal{F},T)\in \operatorname{Good}_1(M)$, we therefore need $i=0$ in which case $\{n\}=F_1$ (so in particular $\{n\}$ is equal to a flat). We have an equality $$\{\mathcal{F}\in \mathcal{L}^M_{(r)}:F_1=\{n\}\}=\{\mathcal{F}'\in \mathcal{L}_{(r)}^{(M/n)\cup \{n\}}:F_1'=\{n\}\},$$
which implies $\operatorname{Good}_1(M)=\operatorname{Good}_1((M/n)\cup \{n\})$. However, every $\mathcal{F}'\in \mathcal{L}^{(M/\{n\})\cup \{n\}}_{(r)}$ has $\{n\}\in F_{i'+1}'\setminus F_{i'}'$ for some $i'$, so $\operatorname{Good}_1(M)=\operatorname{Good}((M/n)\cup \{n\})$. This contributes the term $1_{\{n\}\text{ a flat}}h_{(M/n)\cup \{n\}}$ to the recursion.

Before moving to the second case, we handle the case that $n$ is a coloop, and assume in the future that $n$ does not contain a coloop. Because $n$ is a coloop we have $M=(M/\{n\})\cup \{n\}$, so  by the above reasoning $\operatorname{Good}(M)$ only contains $(\mathcal{F},T)$ with $\{n\}=F_1$. As $n$ is a coloop, we have
$$\{\mathcal{F}'\in \mathcal{L}^{M/n}_{(r-1)}\}\cong \{\mathcal{F}\in \mathcal{L}_{(r)}^{M}:F_1=\{n\}\}$$ obtained by taking $\mathcal{F}'$ to the chain $\{n\}\subset F_1'\cup \{n\}\subset \cdots \subset F_{r-1}'\cup \{n\}$. An element of $(\mathcal{F},T)\in \operatorname{Good}(M)$ is therefore seen to be constructed from an element $(\mathcal{F}',T')\in \operatorname{Good}(M)$ by taking the associated chain $\mathcal{F}\in \mathcal{L}^M_{(r)}$ with $F_1=\{n\}$, and adding onto $T'$ the singleton edge $\{v_n+t_0\}$ on any vertex of $T'$. This always makes a good tree since all $t_j' \ll L$, while $t_0\approx L$ to have $v_n+t_0$ overlap a vertex of $T'$. This shows that $h_M=(z_0+\cdots+z_{\crk(M)})h_{M/\{n\}}$.

In the second case (recalling that we are assuming that $M$ has no coloop), we consider the subset $\operatorname{Cand}_2(M)$ of those $(\mathcal{F},T)$ such that for the $i$ with $n\in F_{i+1}\setminus F_i$, we have $|F_{i+1}\setminus F_i|\ge 2$, and $v_n+t_i$ lies only in $H_i+t_i$. We enumerate the subset $\operatorname{Good}_2(M)=\operatorname{Cand}_2(M)\cap \operatorname{Good}(M)$ by constructing a natural bijection
$$\operatorname{Good}_2(M)\cong \operatorname{Good}(M\setminus n).$$
For $(\mathcal{F},T)\in \operatorname{Cand}_2(M)$ we must have $t_j \approx 0$ for all $j$ including $j=i$, so in $T$ we have a vertex of multiplicity $1$ at $\approx -L$ (the translate of $v_n$), and all other vertices are $\approx 0$. As $n$ is not a coloop in $M$, we have a bijection $$ \mathcal{L}^{M\setminus n}_{(r)}\cong \{\mathcal{F}\in \mathcal{L}^M_{(r)}:\text{for the $i$ such that }n\in F_{i+1}\setminus F_i,\text{ we have }|F_{i+1}\setminus F_i|=2\}.$$ To see this, note that deleting $n$ from any flat in $M$ yields a flat in $M\setminus \{n\}$ by definition, and the condition $|F_{i+1}\setminus F_i|\ge 2$ ensures that all flats $F_i'$ in the resulting chain $\mathcal{F}$ are distinct. Conversely, given a maximal chain $\mathcal{F}'\in \mathcal{L}^{M\setminus (n)}_{(r)}$, each flat in the chain is already a flat in $M$ or becomes a flat after adding $n$. If $F_{i+1}'$ is the first flat which requires $n$, then all future flats require $n$, so $F_1'\subset \cdots \subset F_i'\subset F_{i+1}'\cup \{n\}\subset \cdots \subset F_r'\cup \{n\}$ is the chain of flats in $\mathcal{L}^M_{(r)}$. This induces a bijection
$\operatorname{Cand}_2(M)\cong \operatorname{Cand}(M\setminus n)$ which obviously preserves goodness, so we get a bijection $\operatorname{Good}_2(M)\cong \operatorname{Good}(M\setminus n)$.
Because the element $v_n+t_i$ that we omit under the bijection is always the leftmost vertex of $T$, this contributes $h_{M\setminus n}$ to the recursion.
 
Finally, we will consider $(\mathcal{F},T)\in \operatorname{Good}(M)\setminus (\operatorname{Good}_1(M)\sqcup \operatorname{Good}_2(M))$. The condition is then that for the $i$ with $n\in F_{i+1}\setminus F_i$, we have $|F_{i+1}\setminus F_i|\ge 2$ and $v_n+t_i$ is a vertex of multiplicity at least $2$ in $T$. In this case, we must have $t_i\approx L$, which then implies there is some proper subset $F\subset \{0,\ldots,r\}$ with $t_j\approx L$ for all $j\in F$, and $t_j\ll L$ for all $j\not \in F$. Since the set of all $j$ such that $t_j\ge L/2$ is a flat in $\mathcal{F}$, we conclude that $n\in F\in \mathcal{L}^M_{(1)}$ and $F\in \mathcal{F}$. Note also that if $n$ is a coloop in $F$ then writing $F=F_j$ we have $j\ge i+1$, so $F_{i+1}\setminus F_i=\{n\}$ (being the difference of two consecutive flats in the chain $F_1\subset\cdots\subset F_j=F$ from $\mathcal{L}^{M|_F}_{(\rk(F)-1)}$ containing a coloop), a contradiction. Hence we can partition
$$\operatorname{Good}(M)=\operatorname{Good}_1(M)\sqcup \operatorname{Good}_2(M)\sqcup \bigsqcup_{\substack{n\in F\in \mathcal{L}^M_{(1)}\\n\text{ not a coloop in }F}} \operatorname{Good}_F(M).$$
We claim that there is a bijection
$$\operatorname{Good}_F(M)\cong \operatorname{Good}((M/F) \cup \{n\})\times \operatorname{Good}((M|_F)\setminus \{n\}).$$
Suppose $\rk_M(F)=j$, so in any chain $F\in \mathcal{F}\in \mathcal{L}^M_{(r)}$ we have $F=F_j$. For $k=1,\ldots,j-1$ the sets $F_{k+1}\setminus F_{k}$ for $k=0,\ldots,j-1$ are the differences of consecutive flats in the chain of flats $F_1\subset \cdots\subset F_{j-1}$ in $\mathcal{L}^{M|_F}_{\rk(M|_F)-1}$. This gives a good hypertree for $M|_F \setminus \{n\}$ induced by the rightmost grouping in $\bigcup (H_j+t_j)$ clustered around $L$ (note we omit $n$ since $v_n+t_i$ itself is actually very close to $0$ since $v_n=-L$ and $t_i\approx L$). Next, the $F_{k+1}\setminus F_{k}$ for $k=j,\ldots,r$ are the differences of consecutive flats in the chain of flats $F_{j+1}\setminus F_j \subset \cdots \subset \{0,\ldots,n\}\setminus F_j$ in $\mathcal{L}^{M/F}_{(\rk(M/F)-1)}$, and we get a good hypertree in $M/F$ induced by the leftmost grouping (excluding $v_n+t_i$). Finally, $v_n+t_i$ can freely choose any element of the hypertree for $M/F$ to overlap (the different choices have the effect of shifting the entire $M|_F$ hypertree without changing the intersection pattern or goodness inequalities), and we can model this free choice as taking a good hypertree for $(M/F)\cup \{n\}$ (as was discussed in the coloop case above).

Note now that there is a bijection
$$\mathcal{L}^{M|_F\setminus \{n\}}_{\rk(F)-1}\times \mathcal{L}^{M/F}\cong\{\mathcal{F}\in \mathcal{L}^M_{(r)}: n\in F\text{ with $n$ not a coloop in $F$} \}.$$
obtained by composing the injection $\mathcal{L}_{\rk(F)-1}^{(M|_F) \setminus n}\hookrightarrow \mathcal{L}_{\rk(F)-1}^{M|_F}$ from the second case with the natural identification $\mathcal{L}^{M|_F}_{(\rk(M|_F)-1)}\times \mathcal{L}^{M/F}_{(\rk(M/F)-1)}\cong \{\mathcal{F}\in\mathcal{L}^M:F\in \mathcal{F}\}$ obtained by $(\mathcal{F}',\mathcal{F}'')$ to $\mathcal{F}'\sqcup \{F\}\sqcup \{F\sqcup F'':F''\in \mathcal{F}''\}$. Hence we do obtain every pair of trees in $\operatorname{Good}((M/F)\cup \{n\})\times \operatorname{Good}((M|_F)\setminus \{n\})$ in this way. This contributes $h_{M|_F\setminus n}\mathcal{T}^{\crk(M|_F)}h_{(M/F)\cup \{n\}}$ to the recursion.
\end{proof}
\begin{exmp}
We compute $g_M(2,2,1) = \deg_{A^\bullet(M)}(\delta_{\{2,4\}})$ when $M$ is the rank $3$ matroid realized by the $5$ vectors $v_0=e_1$, $v_1=e_2$, $v_2=e_3$, $v_3=e_1+e_2+e_3$ and $v_4=e_1+e_2$ in $\mathbb{C}^3$. Write $U_{a,b}$ for the matroid associated to a collection of $b$ vectors in an $a$-dimensional vector space with no linear dependencies among any $a-1$ vectors, and $\oplus$ for the matroid direct sum. We have $M\setminus \{4\}=U_{3,4}$, $M/\{4\}=U_{1,2}\oplus U_{1,2}$, and $4$ is not a coloop inside either of the proper flats $F_1=\{0,1,4\}$, $F_2=\{2,3,4\}$ strictly containing $4$, and we note that $M|_{F_1\setminus \{4\}}=M|_{F_2\setminus \{4\}}=U_{1,1}\oplus U_{1,1}$ and $M/F_1=M/F_2=U_{1,2}$.
We compute $h_{U_{3,4}}(z_0,z_1)=3z_0^2+3 z_0z_1+ z_1^2$, $h_{U_{1,2}\oplus U_{1,2}}(z_0,z_1,z_2)= z_0+2 z_1+ z_2$, $h_{U_{1,1}\oplus U_{1,1}}(z_0)=z_0$, and $h_{U_{1,2}}(z_1,z_2)=1$.
Therefore
\begin{align*}h_M&=h_{U_{3,4}}+(z_0+z_1+z_2)h_{U_{1,2}\oplus U_{1,2}}+2(z_1+z_2)h_{U_{1,1}\oplus U_{1,1}}\mathcal{T}h_{U_{1,2}}\\
&=3z_0^2+3z_0z_1+z_1^2+(z_0+z_1+z_2)(z_0+2z_1+z_2)+2z_0(z_1+z_2)\\
&=4z_0^2+3 z_1^2 +z_2^2 +8 z_0z_1+4z_0z_2+3z_1z_2.\end{align*}
Hence $g_M(2,2,1)=[z_0^1z_1^1z_2^0]h_M=8$.
\end{exmp}

\bibliographystyle{alpha} \bibliography{references.bib}

\end{document}